\documentclass[11pt,oneside]{amsart}

\usepackage[margin=1in]{geometry}
\usepackage{amsmath,amsthm,amssymb}
\usepackage{microtype}
\usepackage{tikz}
\usetikzlibrary{arrows.meta}
\usepackage[colorlinks=true,citecolor=red,linkcolor=red,hypertexnames=false]{hyperref}

\numberwithin{equation}{section}
\theoremstyle{plain}
\newtheorem{theorem}{Theorem}[section]
\newtheorem{lemma}[theorem]{Lemma}
\newtheorem{corollary}[theorem]{Corollary}
\newtheorem{proposition}[theorem]{Proposition}

\theoremstyle{definition}
\newtheorem{definition}[theorem]{Definition}
\newtheorem{example}[theorem]{Example}
\newtheorem{question}[theorem]{Question}

\theoremstyle{remark}
\newtheorem{remark}[theorem]{Remark}

\title[Additive energy and Fourier complexity of graphs]{Additive energy and Fourier complexity of graphs}
\author{V. Gupta}
\address{Department of Mathematics, University of Rochester}
\email{vishalgupta@rochester.edu}
\author{S. Ferguson}
\address{Department of Mathematics, New York University}
\email{sjf370@nyu.edu}
\author{A. Iosevich}
\thanks{The third listed author acknowledges partial support from the National Science Foundation under grant DMS-2506858 and from the 2026 Simons Fellowship.}
\address{Department of Mathematics, University of Rochester}
\email{alex.iosevich@rochester.edu}
\author{J. Iosevich}
\address{Department of Applied Mathematics, Rochester Institute of Technology}
\email{joshuaiosevich@gmail.com}
\author{H. Tian}
\address{Department of Mathematics, University of Rochester}
\email{htian6@ur.rochester.edu}
\author{B. Song}
\address{Pittsford Sutherland High School}
\email{pandadragon1624@gmail.com}
\date{July 20, 2026}

\begin{document}

\begin{abstract}
We study the additive energy and the Fourier approximation of graph adjacency matrices labeled by a cyclic group. For a graph with $N$ vertices, $s$ edges and adjacency matrix $M$, the additive energy of its directed edge set is at most $\min\{(2s)^3,N^2\operatorname{tr}(M^4)\}$, and equality in the spectral bound holds exactly when the labeling attains the graph-energy lower bound for the Fourier ratio. Such a labeling also minimizes every best $k$-term Fourier approximation error, whose square equals the tail of the squared adjacency singular values. We prove that graphs of odd prime-power order admit an equality-attaining labeling exactly when they are cyclic Cayley graphs, and a connected example on $15$ vertices shows that equality need not imply vertex transitivity at other orders.

We construct regular graph families with asymptotically maximal normalized additive edge energy and unbounded minimum Fourier ratio, one of them with density tending to one half. For dense graph sequences under uniformly random labelings, the collision dimension differs from the reciprocal squared edge density by a quantity tending to zero in probability. We also prove a quantitative stability estimate for the graph-energy bound and show why its dependence on a spectral scale cannot be omitted, even for regular graphs. For a rank-$m$ orthogonal projector we prove $m\leq D_4\leq\operatorname{FR}^2$, characterize the equality $D_4=m$ by cyclic translation invariance, and obtain a quantitative approximation by a circulant projector.
\end{abstract}

\maketitle
{\setlength{\parskip}{0pt}\tableofcontents}

\section{Introduction}
\label{section:introduction}

Relabeling a graph preserves its adjacency spectrum, but it can change the Fourier transform of its adjacency matrix, and we would like to know what the spectrum tells us about the best Fourier representation. The fourth Fourier moment provides a place to begin, since it counts additive quadruples of directed edges, and a spectral bound for this count singles out the labelings that also minimize every best $k$-term Fourier approximation error.

Let $G=(V,E)$ be a simple undirected graph with $N$ vertices and $s>0$ edges. A labeling is a bijection $\pi:{\mathbb Z}_N\to V$, and it determines the directed edge set
$$
A_\pi=\{(x,y)\in{\mathbb Z}_N^2:\{\pi(x),\pi(y)\}\in E\}
$$
together with the adjacency function $f_\pi=1_{A_\pi}$. Each edge contributes both orientations, so $|A_\pi|=2s$. Let $M_\pi$ be the adjacency matrix with entries $f_\pi(x,y)$. We use counting measure and the unitary Fourier transform
$$
\widehat f(m,n)=\frac{1}{N}\sum_{x,y\in{\mathbb Z}_N}f(x,y)e^{-\frac{2\pi i(mx+ny)}{N}},
$$
and for a nonzero function we define its Fourier ratio and its Fourier collision dimension by
$$
\operatorname{FR}(f)=\frac{\|\widehat f\|_1}{\|f\|_2}
\text{ and }
D_4(f)=\frac{\|f\|_2^4}{\|\widehat f\|_4^4}.
$$
The quantity $\operatorname{FR}(f)^2$ is the numerical sparsity of the Fourier vector \cite{BandeiraLewisMixon2018}, while $D_4(f)$ is the reciprocal collision probability of the probability distribution proportional to $|\widehat f|^2$, which is also the exponential of its order-two R\'enyi entropy \cite{Renyi1961}. Both quantities measure the number of significant Fourier coefficients, and $D_4(f)\leq\operatorname{FR}(f)^2$.

Gupta and Iosevich \cite{GuptaIosevich2026} introduced the edge complexity of a graph as the minimum Fourier ratio over vertex labelings and proved that
$$
\operatorname{FR}_{\min}(G)\geq\frac{\mathcal E(G)}{\sqrt{2s}},
$$
where $\mathcal E(G)=\sum_{j=1}^N|\lambda_j|$ is the graph energy and $\lambda_1,\ldots,\lambda_N$ are the adjacency eigenvalues. The equality cases were characterized in \cite{GuptaIosevichEtAl2026}, where it was shown that a labeling attains equality exactly when the Fourier transform of its adjacency matrix has at most one nonzero entry in each row and column, zero rows and columns being allowed, and that every graph admitting such a labeling is regular. We call such a labeling sharp, and we recall the short proofs of these facts below. The computational problem is already difficult for weighted graphs, since Gupta and Iosevich \cite{GuptaIosevichWeighted2026} proved that approximating weighted edge complexity to a specified inverse-polynomial additive accuracy is NP-hard.

For a finite subset $B$ of an abelian group, its additive energy is
$$
E_+(B)=|\{(b_1,b_2,b_3,b_4)\in B^4:b_1+b_2=b_3+b_4\}|.
$$
Write $E_{\mathrm{edge}}(G,\pi)=E_+(A_\pi)$ and $D_4(G,\pi)=D_4(f_\pi)$. The Fourier expression for additive energy shows that
\begin{equation}
D_4(G,\pi)=\frac{N^2(2s)^2}{E_{\mathrm{edge}}(G,\pi)},
\label{equation:introduction-energy}
\end{equation}
and we prove the bound
\begin{equation}
E_{\mathrm{edge}}(G,\pi)\leq\min\{(2s)^3,N^2\operatorname{tr}(M_\pi^4)\}.
\label{equation:introduction-spectral-cap}
\end{equation}
The first term follows by counting triples, while the second uses the adjacency spectrum, and equality in it holds exactly for sharp labelings. For regular graphs the spectral term is always at least as strong as the elementary term, whereas for irregular graphs either can be smaller. A rank argument then shows that a sharp labeling minimizes every best $k$-term Fourier approximation error at once, the square of that error being the tail of the squared adjacency singular values (Theorem~\ref{theorem:spectral-best-k}).

The equality criterion also leads to structural questions. In Section~\ref{section:classification} we describe its action on Fourier characters and use the arithmetic of the cyclic group to prove a complete classification at odd prime-power orders, where every sharp labeling turns out to be circulant. At order $15$, on the other hand, we construct a connected sharp graph that is not vertex-transitive and hence is not a Cayley graph of any group. For general orders the Fourier description leaves explicit conditions requiring the recovered matrix to have entries in $\{0,1\}$ and zero diagonal.

For a dense graph the constant Fourier coefficient can account for most of the fourth moment while many small coefficients make the Fourier ratio large, and centering the adjacency function removes this constant term. We compute both centered and uncentered quantities for several examples, and we construct connected regular graphs whose normalized additive edge energy tends to one while their minimum Fourier ratio tends to infinity. A second construction has density tending to one half under specified labelings. Under a uniformly random labeling, the collision dimension is asymptotic in probability to the reciprocal squared edge density whenever that density stays bounded away from zero.

We next ask what remains true near equality. For adjacency matrices, relative near-equality in the graph-energy bound yields an approximation by a matrix with at most one nonzero entry in each row and column, but the estimate contains an additional spectral scale, and a regular graph example shows why this dependence is necessary. Orthogonal projectors behave better. For every nonzero rank-$m$ orthogonal projector $\Pi$ we prove
$$
m\leq D_4(\Pi)\leq\operatorname{FR}(\Pi)^2
$$
and find a circulant rank-$m$ orthogonal projector $Q$ satisfying
$$
\frac{\|\Pi-Q\|_{\mathrm F}^2}{\|\Pi\|_{\mathrm F}^2}\leq2\left(1-\frac{m}{D_4(\Pi)}\right),
$$
and the equality $D_4(\Pi)=m$ turns out to be equivalent to cyclic translation invariance of the range of $\Pi$.

Our approximation bounds count the Fourier frequencies needed to represent an adjacency function to a prescribed accuracy. To connect these bounds with recovery from observed adjacency entries one must specify how the entries are observed and prove what the observations determine, and we refer to \cite{AldalehEtAl2025,BursteinIosevichNathan2026,GuptaIosevich2026} for recovery results in the Fourier-ratio setting. We also apply the Reiher--Schoen form of the Balog--Szemer\'edi--Gowers theorem \cite{ReiherSchoen2024} to obtain a large subset of the directed edge set with a small difference set.

All norms on functions use counting measure. We write $X\asymp Y$ for two-sided bounds with absolute constants, $X\asymp_\eta Y$ when the constants may depend on $\eta$, and $X\sim Y$ when $\frac{X}{Y}\to1$. Weighted additive energy is denoted by $\mathcal E_+(f)$ and graph energy by $\mathcal E(G)$.

\section{Fourier analysis, energy, and collision dimension}
\label{section:finite-fourier}

Let $H$ be a finite abelian group of order $M$. A character of $H$ is a homomorphism from $H$ to the complex numbers of modulus one, the characters form the dual group $\widehat H$, and after normalization by $M^{-\frac{1}{2}}$ they form an orthonormal basis under counting measure. Our Fourier conventions are therefore
\begin{equation}
\widehat f(\gamma)=M^{-\frac{1}{2}}\sum_{x\in H}f(x)\overline{\gamma(x)},
\label{equation:finite-fourier-transform}
\end{equation}
\begin{equation}
f(x)=M^{-\frac{1}{2}}\sum_{\gamma\in\widehat H}\widehat f(\gamma)\gamma(x),
\label{equation:finite-fourier-inversion}
\end{equation}
and
\begin{equation}
\|\widehat f\|_2=\|f\|_2.
\label{equation:plancherel}
\end{equation}
See \cite{Terras1999,TaoVu2006} for character orthogonality and Fourier analysis on finite groups.

\begin{definition}
For $f\neq0$, its Fourier ratio and Fourier collision dimension are
$$
\operatorname{FR}(f)=\frac{\|\widehat f\|_1}{\|\widehat f\|_2}
\text{ and }
D_4(f)=\frac{\|\widehat f\|_2^4}{\|\widehat f\|_4^4}.
$$
\end{definition}

For a nonzero vector $v$, its numerical sparsity is $\operatorname{ns}(v)=\frac{\|v\|_1^2}{\|v\|_2^2}$, so that $\operatorname{FR}(f)^2=\operatorname{ns}(\widehat f)$ \cite{BandeiraLewisMixon2018}. To understand the collision dimension, consider the probability distribution
$$
p_f(\gamma)=\frac{|\widehat f(\gamma)|^2}{\|f\|_2^2}.
$$
Two independent samples from this distribution agree with probability $\sum_\gamma p_f(\gamma)^2$, and the reciprocal of this probability is
$$
D_4(f)=\frac{1}{\sum_\gamma p_f(\gamma)^2}=e^{H_2(p_f)},
$$
where $H_2(p)=-\log\sum_\gamma p(\gamma)^2$ is the order-two R\'enyi entropy \cite{Renyi1961}. For example, if exactly $r$ Fourier coefficients are nonzero and they all have the same modulus, then $D_4(f)=r$.

For the unnormalized convolution $(f*g)(z)=\sum_xf(x)g(z-x)$, the substitution $z=x+y$ in the Fourier sum separates the two variables, and we obtain
\begin{equation}
\widehat{f*g}(\gamma)
=\frac{1}{\sqrt M}\left(\sum_xf(x)\overline{\gamma(x)}\right)\left(\sum_yg(y)\overline{\gamma(y)}\right)
=\sqrt M\widehat f(\gamma)\widehat g(\gamma).
\label{equation:convolution-transform}
\end{equation}

\begin{definition}
The weighted additive energy of $f:H\to{\mathbb C}$ is
$$
\mathcal E_+(f)=\sum_{\substack{x+y=x'+y'\\x,y,x',y'\in H}}f(x)f(y)\overline{f(x')f(y')}.
$$
For a subset $A\subset H$, its additive energy is
$$
E_+(A)=|\{(a_1,a_2,a_3,a_4)\in A^4:a_1+a_2=a_3+a_4\}|,
$$
so that $E_+(A)=\mathcal E_+(1_A)$.
\end{definition}

\begin{proposition}[Energy identity]
\label{proposition:general-energy-identity}
For every $f:H\to{\mathbb C}$,
$$
\mathcal E_+(f)=\|f*f\|_2^2=M\sum_\gamma|\widehat f(\gamma)|^4.
$$
In particular, if $f\neq0$, then
$$
D_4(f)=\frac{M\|f\|_2^4}{\mathcal E_+(f)}.
$$
\end{proposition}

\begin{proof}
Grouping the terms of $\mathcal E_+(f)$ according to the common value $z$ of $x+y$ and $x'+y'$ turns the sum into
$$
\|f*f\|_2^2=\sum_z\left|\sum_{x+y=z}f(x)f(y)\right|^2=\mathcal E_+(f),
$$
and Plancherel together with the equation \eqref{equation:convolution-transform} then shows that
$$
\|f*f\|_2^2=\|\widehat{f*f}\|_2^2=M\sum_\gamma|\widehat f(\gamma)|^4.
$$
The last assertion follows from the definition of $D_4$ and Plancherel.
\end{proof}

\begin{theorem}[Collision hierarchy]
\label{theorem:collision-hierarchy-general}
For every nonzero $f:H\to{\mathbb C}$,
$$
1\leq D_4(f)\leq\operatorname{FR}(f)^2\leq|\operatorname{supp}(\widehat f)|\leq M.
$$
Equality $D_4(f)=\operatorname{FR}(f)^2$ holds if and only if all nonzero Fourier coefficients have the same modulus.
\end{theorem}

\begin{proof}
Set $a_\gamma=|\widehat f(\gamma)|$. The first inequality is $\|a\|_4\leq\|a\|_2$, and H\"older's inequality with the exponents $\frac{3}{2}$ and $3$ yields
$$
\sum_\gamma a_\gamma^2=\sum_\gamma a_\gamma^{\frac{2}{3}}a_\gamma^{\frac{4}{3}}
\leq\left(\sum_\gamma a_\gamma\right)^{\frac{2}{3}}\left(\sum_\gamma a_\gamma^4\right)^{\frac{1}{3}},
$$
which after cubing and dividing becomes $D_4(f)\leq\operatorname{FR}(f)^2$. Equality in H\"older's inequality requires $a_\gamma$ and $a_\gamma^4$ to be proportional where they are positive, and this is the equal-modulus condition. Finally, the Cauchy--Schwarz inequality on the support shows that $\|\widehat f\|_1\leq|\operatorname{supp}(\widehat f)|^{\frac{1}{2}}\|\widehat f\|_2$, and since $|\widehat H|=M$ the remaining inequalities follow.
\end{proof}

\begin{corollary}
\label{corollary:energy-fr-general}
For every nonzero $f:H\to{\mathbb C}$,
$$
\operatorname{FR}(f)\geq\frac{\sqrt M\|f\|_2^2}{\sqrt{\mathcal E_+(f)}}.
$$
\end{corollary}

\begin{proof}
Substitute the energy identity into $D_4(f)\leq\operatorname{FR}(f)^2$ and take square roots.
\end{proof}

\begin{proposition}[Indicator uncertainty]
\label{proposition:indicator-uncertainty-general}
If $A\subset H$ is nonempty, then $E_+(A)\leq|A|^3$ and
$$
\operatorname{FR}(1_A)^2\geq D_4(1_A)=\frac{M|A|^2}{E_+(A)}\geq\frac{M}{|A|}.
$$
\end{proposition}

\begin{proof}
The first three entries of an additive quadruple determine the fourth through $a_4=a_1+a_2-a_3$, so there are at most $|A|^3$ such quadruples. The remaining statements follow from the energy identity, from $\|1_A\|_2^2=|A|$, and from the collision hierarchy.
\end{proof}

\section{Additive edge energy}
\label{section:edge-energy}

A labeling turns the adjacency relation of a graph into a subset of ${\mathbb Z}_N^2$, and the energy identity of the preceding section then relates additive quadruples of edges to the Fourier coefficients of the adjacency matrix. Throughout this section $G=(V,E)$ is a finite simple undirected graph with $N$ vertices and $s>0$ edges.

\begin{definition}
A labeling of $G$ is a bijection $\pi:{\mathbb Z}_N\longrightarrow V$. The directed edge set associated with $\pi$ is
$$
A_\pi=\{(x,y)\in{\mathbb Z}_N^2:\{\pi(x),\pi(y)\}\in E\},
$$
and the adjacency function associated with $\pi$ is $f_\pi=1_{A_\pi}$.
\end{definition}

Because $G$ is undirected, $(x,y)\in A_\pi$ if and only if $(y,x)\in A_\pi$, and since there are no loops, $(x,x)\notin A_\pi$. Each undirected edge contributes exactly two ordered pairs, so $|A_\pi|=2s$.

On ${\mathbb Z}_N^2$ the unitary Fourier transform takes the form
\begin{equation}
\widehat f(m,n)=\frac{1}{N}\sum_{x,y\in{\mathbb Z}_N}f(x,y)e^{-\frac{2\pi i(mx+ny)}{N}},
\label{equation:two-dimensional-transform}
\end{equation}
the normalizing factor being $\frac{1}{N}$ because $|{\mathbb Z}_N^2|=N^2$, and consequently $\|\widehat f\|_2=\|f\|_2$. It is useful to write this transform as a matrix product. Let
$$
U_{m,x}=\frac{1}{\sqrt N}e^{-\frac{2\pi imx}{N}}
$$
be the unitary Fourier matrix, so that $U^*U=UU^*=I$, and let $M_\pi=(f_\pi(x,y))_{x,y\in{\mathbb Z}_N}$ be the adjacency matrix under the labeling $\pi$. Since $U$ is symmetric, the formula \eqref{equation:two-dimensional-transform} says that
\begin{equation}
\widehat M_\pi=UM_\pi U,
\label{equation:matrix-fourier-transform}
\end{equation}
where the entries of $\widehat M_\pi$ are the Fourier coefficients of $f_\pi$.

For a matrix $B$ we write $\|B\|_{\mathrm F}$ for the square root of the sum of its squared entry moduli, $\|B\|_{\mathrm{op}}$ for its largest singular value, and $\|B\|_*$ for the sum of its singular values. These are the Frobenius, operator and nuclear norms, and in particular the entrywise function norm $\|B\|_2$ equals $\|B\|_{\mathrm F}$. We write $\operatorname{tr}(B)$ for the sum of the diagonal entries of a square matrix, and we shall use repeatedly the fact that left and right multiplication by unitary matrices preserve singular values.

\begin{definition}
The Fourier ratio and the Fourier collision dimension of $G$ under the labeling $\pi$ are
$$
\operatorname{FR}(G,\pi)=\operatorname{FR}(f_\pi)=\frac{\|\widehat f_\pi\|_1}{\|\widehat f_\pi\|_2}
\text{ and }
D_4(G,\pi)=D_4(f_\pi)=\frac{\|\widehat f_\pi\|_2^4}{\|\widehat f_\pi\|_4^4}.
$$
The minimum Fourier ratio is
$$
\operatorname{FR}_{\min}(G)=\min_{\pi}\operatorname{FR}(G,\pi),
$$
where the minimum is taken over all labelings of $G$ by ${\mathbb Z}_N$.
\end{definition}

\begin{definition}
The additive edge energy of $G$ under the labeling $\pi$ is
$$
E_{\mathrm{edge}}(G,\pi)=E_+(A_\pi)=|\{(a_1,a_2,a_3,a_4)\in A_\pi^4:a_1+a_2=a_3+a_4\}|,
$$
where the addition is coordinatewise in ${\mathbb Z}_N^2$.
\end{definition}

\begin{theorem}[Edge-energy identity]
\label{theorem:edge-energy-identity}
Let $G$ be a simple graph with $N$ vertices and $s>0$ edges, and let $\pi$ be a labeling. Then
$$
E_{\mathrm{edge}}(G,\pi)=N^2\sum_{m,n\in{\mathbb Z}_N}|\widehat f_\pi(m,n)|^4
$$
and
\begin{equation}
D_4(G,\pi)=\frac{N^2(2s)^2}{E_{\mathrm{edge}}(G,\pi)}.
\label{equation:graph-collision-dimension}
\end{equation}
\end{theorem}

\begin{proof}
Apply Proposition~\ref{proposition:general-energy-identity} to $H={\mathbb Z}_N^2$ and $f=1_{A_\pi}$, for which $|H|=N^2$ and $\|f\|_2^2=|A_\pi|=2s$.
\end{proof}

\begin{corollary}[Energy obstruction for the Fourier ratio]
\label{corollary:graph-energy-obstruction-fr}
Under the hypotheses of Theorem~\ref{theorem:edge-energy-identity},
\begin{equation}
\operatorname{FR}(G,\pi)\geq\frac{2Ns}{\sqrt{E_{\mathrm{edge}}(G,\pi)}}
\label{equation:edge-energy-fr-lower}
\end{equation}
and
\begin{equation}
\operatorname{FR}(G,\pi)\geq\frac{N}{\sqrt{2s}}.
\label{equation:graph-uncertainty-lower}
\end{equation}
\end{corollary}

\begin{proof}
The first inequality follows from $D_4(G,\pi)\leq\operatorname{FR}(G,\pi)^2$ and \eqref{equation:graph-collision-dimension}. For the second, apply Proposition~\ref{proposition:indicator-uncertainty-general} to the set $A_\pi\subset{\mathbb Z}_N^2$, where $|{\mathbb Z}_N^2|=N^2$ and $|A_\pi|=2s$, to obtain $\operatorname{FR}(G,\pi)^2\geq\frac{N^2}{2s}$, and take square roots.
\end{proof}

Changing the labeling can change the additive edge energy, and maximizing over labelings removes this dependence.

\begin{definition}
Define
$$
E_{\mathrm{edge}}^{\max}(G)=\max_\pi E_{\mathrm{edge}}(G,\pi),
$$
where the maximum is over all labelings $\pi:{\mathbb Z}_N\to V$.
\end{definition}

\begin{theorem}[Labeling-independent energy bound]
\label{theorem:labeling-independent-energy-bound}
Let $G$ be a simple graph with $N$ vertices and $s>0$ edges. Then
\begin{equation}
\operatorname{FR}_{\min}(G)\geq\frac{2Ns}{\sqrt{E_{\mathrm{edge}}^{\max}(G)}},
\label{equation:fr-min-edge-energy}
\end{equation}
and since $E_{\mathrm{edge}}^{\max}(G)\leq(2s)^3$, this implies $\operatorname{FR}_{\min}(G)\geq\frac{N}{\sqrt{2s}}$.
\end{theorem}

\begin{proof}
For every $\pi$ we have
$$
\operatorname{FR}(G,\pi)\geq\frac{2Ns}{\sqrt{E_{\mathrm{edge}}(G,\pi)}}\geq\frac{2Ns}{\sqrt{E_{\mathrm{edge}}^{\max}(G)}},
$$
and it remains to take the minimum over labelings.
\end{proof}

We now recall a second labeling-independent lower bound from \cite{GuptaIosevich2026}. If $\lambda_1,\ldots,\lambda_N$ are the eigenvalues of the adjacency matrix of $G$, its graph energy is
$$
\mathcal E(G)=\sum_{j=1}^N|\lambda_j|.
$$
Graph energy was introduced by Gutman \cite{Gutman1978}, and we refer to \cite{LiShiGutman2012} for its theory. The following elementary matrix inequality explains its appearance here, and the inequality together with its equality case appears as Lemma 3 in Appendix C of Zhu, Hayashi and Chen \cite{ZhuHayashiChen2018}.

\begin{lemma}[Entrywise norm and singular values]
\label{lemma:nuclear-entrywise}
Let $B$ be a complex square matrix and let $\|B\|_*$ be the sum of its singular values. Then
$$
\|B\|_*\leq\sum_{i,j}|B_{ij}|,
$$
and equality holds if and only if each row and each column of $B$ has at most one nonzero entry.
\end{lemma}

\begin{proof}
Choose a singular value decomposition $B=V\Sigma W^*$ with $V,W$ unitary, and put $Q=VW^*$. Then
$$
\|B\|_*=\operatorname{Re}\operatorname{tr}(Q^*B)=\operatorname{Re}\sum_{i,j}\overline{Q_{ij}}B_{ij}\leq\sum_{i,j}|B_{ij}|,
$$
since every entry of a unitary matrix has modulus at most one. If equality holds, then each nonzero $B_{ij}$ requires $|Q_{ij}|=1$, and because a row or column of a unitary matrix has squared norm one, it contains at most one such entry, so the same support restriction holds for $B$. Conversely, if $B$ has the stated support, then its nonzero singular values are the moduli of its nonzero entries, whose sum is the entrywise $\ell^1$ norm.
\end{proof}

\begin{theorem}[Graph-energy bound, \cite{GuptaIosevich2026}]
\label{theorem:prior-graph-energy-bound}
Let $G$ be a simple graph with $s>0$ edges. Then
$$
\operatorname{FR}_{\min}(G)\geq\frac{\mathcal E(G)}{\sqrt{2s}}.
$$
\end{theorem}

\begin{proof}
Apply Lemma~\ref{lemma:nuclear-entrywise} to $UM_\pi U$. Unitary multiplication preserves singular values, so its nuclear norm is $\mathcal E(G)$, whereas its Frobenius norm is $\sqrt{2s}$. Dividing and minimizing over labelings proves the theorem.
\end{proof}

We therefore have two lower bounds, one involving the adjacency spectrum and the other involving additive quadruples of edges.

\begin{corollary}
\label{corollary:three-way-graph-bound}
For every simple graph with $N$ vertices and $s>0$ edges,
$$
\operatorname{FR}_{\min}(G)\geq\max\left\{\frac{\mathcal E(G)}{\sqrt{2s}},\frac{2Ns}{\sqrt{E_{\mathrm{edge}}^{\max}(G)}}\right\}.
$$
\end{corollary}

\begin{remark}
The bound $E_{\mathrm{edge}}^{\max}(G)\leq(2s)^3$ shows that the additive term also contains the support-size uncertainty bound $\operatorname{FR}_{\min}(G)\geq\frac{N}{\sqrt{2s}}$.
\end{remark}

For a fixed labeling we can compare these bounds directly with the number of nonzero Fourier coefficients, since the graph-energy bound and Theorem~\ref{theorem:collision-hierarchy-general} imply
\begin{equation}
\max\left\{\frac{\mathcal E(G)^2}{2s},D_4(G,\pi)\right\}\leq\operatorname{FR}(G,\pi)^2\leq|\operatorname{supp}(\widehat f_\pi)|.
\label{equation:two-effective-dimensions}
\end{equation}
The first term on the left depends only on the graph, while the second can change with the labeling, through the additive energy of its directed edge set.

\subsection{Sharp labelings}
\label{subsection:energy-extremizers}

The equality criterion from \cite{GuptaIosevichEtAl2026} will identify the labelings that attain the spectral upper bound for additive edge energy, and we recall it by means of Lemma~\ref{lemma:nuclear-entrywise}.

\begin{definition}
A labeling $\pi$ of a graph $G$ with $s>0$ edges is called sharp if
$$
\operatorname{FR}(G,\pi)=\frac{\mathcal E(G)}{\sqrt{2s}}.
$$
\end{definition}

\begin{theorem}[Equality criterion, \cite{GuptaIosevichEtAl2026}]
\label{theorem:prior-equality-criterion}
Let $M_\pi$ be the adjacency matrix under a fixed labeling. The labeling is sharp if and only if every row and every column of $\widehat M_\pi=UM_\pi U$ contains at most one nonzero entry.
\end{theorem}

\begin{proof}
Equality in the graph-energy bound is exactly equality in Lemma~\ref{lemma:nuclear-entrywise} for $B=UM_\pi U$.
\end{proof}

A matrix with at most one nonzero entry in every row and every column is called a weighted partial permutation matrix. Empty rows and columns are allowed, and the nonzero entries may have arbitrary weights. The equality criterion says that a sharp labeling has a Fourier transform of precisely this form.

\begin{theorem}[Regularity of sharp graphs, \cite{GuptaIosevichEtAl2026}]
\label{theorem:prior-extremizer-regularity}
Every graph with at least one edge that admits a sharp labeling is regular.
\end{theorem}

\begin{proof}
The coefficient $\widehat M_\pi(0,0)=\frac{2s}{N}$ is nonzero, so the equality criterion forces $\widehat M_\pi(m,0)=0$ for $m\neq0$. If $d(x)$ is the degree of the vertex labeled $x$, then
$$
\widehat M_\pi(m,0)=\frac{1}{N}\sum_xd(x)e^{-\frac{2\pi imx}{N}},
$$
so every nonconstant Fourier coefficient of the degree function vanishes, and Fourier inversion shows that all degrees are equal.
\end{proof}

\begin{theorem}[A sharp spectral upper bound for additive edge energy]
\label{theorem:energy-spectral-cap}
For every labeling of a graph with adjacency matrix $M$ and $s>0$ edges,
$$
E_{\mathrm{edge}}(G,\pi)\leq\min\{(2s)^3,N^2\operatorname{tr}(M^4)\}.
$$
Equality in the spectral inequality $E_{\mathrm{edge}}(G,\pi)\leq N^2\operatorname{tr}(M^4)$ holds if and only if $\pi$ is sharp. Consequently,
$$
D_4(G,\pi)\geq\max\left\{\frac{N^2}{2s},\frac{(2s)^2}{\operatorname{tr}(M^4)}\right\},
$$
and $E_{\mathrm{edge}}^{\max}(G)=N^2\operatorname{tr}(M^4)$ if and only if $G$ admits a sharp labeling.
\end{theorem}

\begin{proof}
For any complex matrix $B$ write $r_i=\sum_j|B_{ij}|^2$ for its squared row norms. Then
\begin{equation}
\sum_{i,j}|B_{ij}|^4\leq\sum_ir_i^2\leq\operatorname{tr}((BB^*)^2).
\label{equation:entrywise-fourth-spectral}
\end{equation}
The first inequality follows by expanding each squared row norm, and the second from $\operatorname{tr}((BB^*)^2)=\sum_{i,k}|(BB^*)_{ik}|^2$ together with $(BB^*)_{ii}=r_i$. Equality in the first inequality requires at most one nonzero entry in each row, equality in the second requires distinct rows to be orthogonal, and together these conditions say that the nonzero entries occupy distinct rows and columns, which is also sufficient for equality.

Apply this to $B=UM_\pi U$. Its singular values are those of $M$, so $\operatorname{tr}((BB^*)^2)=\operatorname{tr}(M^4)$, and the edge-energy identity proves the spectral bound, whose equality case is identified by Theorem~\ref{theorem:prior-equality-criterion}. Proposition~\ref{proposition:indicator-uncertainty-general} supplies the bound $(2s)^3$, and dividing by the two upper bounds produces the inequalities for $D_4$. Finally, there are only finitely many labelings, so the maximum is attained, and the equality criterion therefore also proves the last assertion.
\end{proof}

\begin{proposition}[Comparison of the two energy caps]
\label{proposition:cap-comparison}
If $G$ is $d$-regular with $d>0$, then
$$
N^2\operatorname{tr}(M^4)\leq(Nd)^3=(2s)^3,
$$
and equality holds if and only if every connected component is $K_{d,d}$. For irregular graphs either cap can be smaller.
\end{proposition}

\begin{proof}
Since $|\lambda_j|\leq d$ and $\sum_j\lambda_j^2=Nd$,
$$
\operatorname{tr}(M^4)=\sum_j\lambda_j^4\leq d^2\sum_j\lambda_j^2=Nd^3.
$$
Equality forces every nonzero eigenvalue to have modulus $d$, and the componentwise argument in Corollary~\ref{corollary:extremizer-dimensions-agree}, which uses only regularity and the spectrum, then shows that every component is $K_{d,d}$. Conversely, these components have only $0$, $d$ and $-d$ as eigenvalues. For the star $K_{1,n}$ the ratio of the spectral cap to the elementary cap is
$$
\frac{(n+1)^22n^2}{(2n)^3}=\frac{(n+1)^2}{4n},
$$
which tends to infinity, whereas $K_4$ with one edge deleted has $2s=10$ and $\operatorname{tr}(M^4)=50$, so that its spectral cap is $800<1000=(2s)^3$.
\end{proof}

\begin{proposition}[Higher additive moments]
\label{proposition:higher-additive-moments}
For an integer $k\geq2$, define
$$
E_k(A)=|\{(a_1,\ldots,a_k,b_1,\ldots,b_k)\in A^{2k}:a_1+\cdots+a_k=b_1+\cdots+b_k\}|.
$$
For every labeling of a graph with at least one edge and adjacency matrix $M$,
\begin{equation}
E_k(A_\pi)=N^{2k-2}\sum_{m,n}|\widehat M_\pi(m,n)|^{2k}\leq N^{2k-2}\operatorname{tr}(M^{2k}),
\label{equation:higher-additive-moments}
\end{equation}
and equality holds if and only if the labeling is sharp. In particular, a labeling attains one of these higher-moment bounds if and only if it attains all of them.
\end{proposition}

\begin{proof}
The convolution identity and Plancherel give the first equality, since $E_k(A_\pi)$ is the squared $\ell^2$ norm of the $k$-fold convolution of $1_{A_\pi}$ with itself. For a complex matrix $B$, put $X=BB^*$ and $r_i=X_{ii}$. Then
$$
\sum_{i,j}|B_{ij}|^{2k}\leq\sum_ir_i^k\leq\operatorname{tr}(X^k),
$$
where the first inequality follows by expanding the integer power of each squared row norm, and the second by diagonalizing $X=V\operatorname{diag}(\mu_a)V^*$ and applying Jensen's inequality to $r_i=\sum_a|V_{ia}|^2\mu_a$, using $\sum_a|V_{ia}|^2=1$. Equality in the first inequality requires each row of $B$ to have at most one nonzero entry. Since $t\mapsto t^k$ is strictly convex on $[0,\infty)$, equality in the second says that each standard basis vector belongs to a single eigenspace of $X$, so that $X$ is diagonal. Together, the two equality conditions say exactly that $B$ is a weighted partial permutation matrix. Apply this with $B=\widehat M_\pi$, whose singular values are the numbers $|\lambda_j(M)|$.
\end{proof}

\begin{theorem}[Exact additive energy of a sharp labeling]
\label{theorem:extremizer-exact-energy}
Let $G$ be a simple graph with $N$ vertices and $s>0$ edges, let $M$ be its adjacency matrix under a sharp labeling $\pi$, and let $\lambda_1,\ldots,\lambda_N$ be the adjacency eigenvalues. Then
\begin{equation}
E_{\mathrm{edge}}(G,\pi)=N^2\operatorname{tr}(M^4)=N^2\sum_{j=1}^N|\lambda_j|^4,
\label{equation:extremizer-edge-energy}
\end{equation}
\begin{equation}
D_4(G,\pi)=\frac{(2s)^2}{\operatorname{tr}(M^4)},
\label{equation:extremizer-collision-dimension}
\end{equation}
and
\begin{equation}
|\operatorname{supp}(\widehat M)|=\operatorname{rank}(M).
\label{equation:extremizer-support-rank}
\end{equation}
Moreover, the nonzero values $|\widehat M(m,n)|$, counted with multiplicity, are exactly the nonzero values $|\lambda_j|$.
\end{theorem}

\begin{proof}
The nonzero entries of $\widehat M$ occupy distinct rows and columns, so a column permutation makes $\widehat M$ diagonal, and their moduli are therefore the nonzero singular values of $\widehat M$, which equal those of $M$, namely the nonzero $|\lambda_j|$. Summing fourth powers and using the edge-energy identity proves the first formula, and division yields the second. The number of these entries equals $\operatorname{rank}(\widehat M)=\operatorname{rank}(M)$, which proves the last assertion.
\end{proof}

For a sharp labeling, the equality between collision dimension and squared Fourier ratio can now be read from the adjacency spectrum.

\begin{corollary}
\label{corollary:extremizer-dimensions-agree}
Under the hypotheses of Theorem~\ref{theorem:extremizer-exact-energy},
$$
D_4(G,\pi)=\operatorname{FR}(G,\pi)^2
$$
if and only if all nonzero adjacency eigenvalues have the same absolute value. By Theorem~\ref{theorem:prior-extremizer-regularity}, $G$ is regular, and if its common degree is $d$, then the equality above holds if and only if every connected component of $G$ is isomorphic to the balanced complete bipartite graph $K_{d,d}$.
\end{corollary}

\begin{proof}
Let $a_1,\ldots,a_r$ be the nonzero values $|\widehat M(m,n)|$. The inequality
$$
\left(\sum_ja_j^2\right)^3\leq\left(\sum_ja_j\right)^2\left(\sum_ja_j^4\right)
$$
is H\"older's inequality applied to $a_j^2=a_j^{\frac{2}{3}}a_j^{\frac{4}{3}}$, and equality holds precisely when all the positive $a_j$ are equal. By Theorem~\ref{theorem:extremizer-exact-energy} these numbers are the nonzero absolute adjacency eigenvalues, which proves the first assertion.

Now assume that $G$ is $d$-regular, and let $C$ be one of its connected components, which is connected and $d$-regular, so that its eigenvalue $d$ is simple. A connected $d$-regular graph has eigenvalue $-d$ if and only if it is bipartite, in which case $-d$ is also simple, see for example \cite{BrouwerHaemers2012}. To see the connection, bipartiteness supplies the eigenvector that is $1$ on one vertex class and $-1$ on the other, while the converse follows from the equality case of the triangle inequality applied to an eigenvector for $-d$.

All nonzero eigenvalues of $C$ have absolute value $d$. Since the trace of its adjacency matrix is zero, the positive eigenvalue $d$ cannot be the only nonzero eigenvalue, so $-d$ occurs, and the preceding spectral fact then shows that $C$ is bipartite and that its only nonzero eigenvalues are the simple pair $d,-d$. Thus its adjacency matrix has rank two. Writing that adjacency matrix in bipartite block form
$$
\begin{pmatrix}0&B\\B^T&0\end{pmatrix},
$$
the rank of the block matrix is $2\operatorname{rank}(B)$, because its image is the direct sum of the images of $B$ and $B^T$, and since the adjacency matrix of $C$ has rank two we obtain $\operatorname{rank}(B)=1$. Because $C$ is connected, $B$ has no zero row and no zero column. Every row of a rank-one matrix is a scalar multiple of any fixed nonzero row, and since the rows of $B$ are nonzero zero-one vectors, that scalar must be one, so that all rows are equal, and the absence of zero columns then forces every entry to be one. Thus $C$ is complete bipartite, and since it is $d$-regular, both vertex classes have size $d$, so that $C\cong K_{d,d}$. Conversely, a disjoint union of copies of $K_{d,d}$ has only the nonzero eigenvalues $d$ and $-d$, so the first part of the corollary yields equality.
\end{proof}

\subsection{Symmetries of the labeling search}

\begin{proposition}[Affine and graph-automorphism invariance]
\label{proposition:affine-invariance}
Let $\phi(x)=ax+b$ with $a\in{\mathbb Z}_N^\times$ and $b\in{\mathbb Z}_N$. Then
$$
\widehat f_{\pi\circ\phi}(m,n)=e^{\frac{2\pi i(m+n)a^{-1}b}{N}}\widehat f_\pi(a^{-1}m,a^{-1}n),
$$
so $\operatorname{FR}$, $D_4$ and the additive edge energy are unchanged, and the same is true of the best Fourier approximation errors and of the quantity $\kappa_\eta$ defined in Section~\ref{section:approximation}. If $\tau$ is a graph automorphism, meaning a permutation of $V$ that preserves adjacency, then $f_{\tau\circ\pi}=f_\pi$.
\end{proposition}

\begin{proof}
Substituting $u=ax+b$ and $v=ay+b$ in the Fourier sum produces the displayed formula, which permutes the coefficient moduli, and a graph automorphism preserves adjacency itself.
\end{proof}

Choose a reference ordering of the vertices, so that labelings are identified with permutations in $S_N$, let $\operatorname{Aut}(G)$ be the group of graph automorphisms, and let $\operatorname{Aff}({\mathbb Z}_N)$ be the group of maps $x\mapsto ax+b$ with $a$ a unit. Labelings related by an automorphism on the left and an affine map on the right have the same quantities above, so it suffices to choose one representative from each double coset in
$$
\operatorname{Aut}(G)\backslash S_N/\operatorname{Aff}({\mathbb Z}_N).
$$
The right action on labelings is free and by itself reduces the search to $\frac{N!}{N\varphi(N)}$ representatives, where $\varphi(N)$ is Euler's totient function, the number of units in ${\mathbb Z}_N$. The further reduction from graph automorphisms depends on the intersections of the two stabilizers, so the sizes of the double cosets can vary.

\section{Structural classification of sharp graphs}
\label{section:classification}

A sharp adjacency matrix sends each Fourier character either to zero or to a scalar multiple of one character, and since the adjacency matrix is real and symmetric, these characters must be paired in a particular way. We begin with the resulting normal form and its consequence for even walks from \cite{GuptaIosevichEtAl2026}.

Recall that a matrix is circulant if its $(x,y)$ entry depends only on $x-y$ in ${\mathbb Z}_N$, or equivalently if it is diagonal in the Fourier basis. Put $\zeta=e^{\frac{2\pi i}{N}}$,
$$
\chi_k(x)=N^{-\frac{1}{2}}\zeta^{kx}
\text{ and }
B=UMU^*.
$$
The matrices $B$ and $\widehat M=UMU$ differ by a column permutation, so the equality criterion applies equally to $B$, and $B$ is Hermitian, which makes the pairing of the Fourier characters explicit.

\begin{theorem}[Fourier involution normal form, \cite{GuptaIosevichEtAl2026}]
\label{theorem:sharp-involution}
A simple graph with at least one edge has a sharp labeling on ${\mathbb Z}_N$ if and only if, in that labeling, there are a set $\Omega\subseteq{\mathbb Z}_N$, an involution $\tau:\Omega\to\Omega$, meaning that $\tau^2(k)=k$, and nonzero numbers $c_k$ for $k\in\Omega$, such that
\begin{equation}
M\chi_k=c_k\chi_{\tau(k)}\text{ for }k\in\Omega
\text{ and }
M\chi_k=0\text{ for }k\notin\Omega.
\label{equation:sharp-character-action}
\end{equation}
For a sharp labeling these data necessarily satisfy
\begin{equation}
0\in\Omega,\tau(0)=0,c_0=d>0,c_{\tau(k)}=\overline{c_k},
\label{equation:sharp-hermitian-data}
\end{equation}
\begin{equation}
-\Omega=\Omega,\tau(-k)=-\tau(k),c_{-k}=\overline{c_k},
\label{equation:sharp-real-data}
\end{equation}
where $G$ is $d$-regular, and moreover
\begin{equation}
M(x,y)=\frac{1}{N}\sum_{k\in\Omega}c_k\zeta^{\tau(k)x-ky}.
\label{equation:sharp-inverse-normal-form}
\end{equation}
Each fixed point of $\tau$ contributes the real eigenvalue $c_k$, and each two-element orbit $\{k,\tau(k)\}$ contributes the eigenvalues $|c_k|$ and $-|c_k|$.
\end{theorem}

\begin{proof}
For a sharp labeling, let $\Omega$ be the set of nonzero columns of $B$, whose indices we call the active frequencies. Each active column $k$ has its unique nonzero entry, say $c_k$, in some row $\tau(k)$, and since the entries of column $k$ of $B$ are the coefficients of $M\chi_k$ in the character basis, this is the character action \eqref{equation:sharp-character-action}. Hermitian symmetry forces $\tau^2(k)=k$ and $c_{\tau(k)}=\overline{c_k}$, and conjugating the character action produces \eqref{equation:sharp-real-data}, because $M$ is real. The $(0,0)$ entry of $B$ is the average degree, which is positive, so all other entries of its zeroth row and column vanish, and consequently $M\chi_0=d\chi_0$, which proves regularity and the assertions at zero. Fourier inversion then yields \eqref{equation:sharp-inverse-normal-form}. A two-element orbit contributes the Hermitian block
$$
\begin{pmatrix}0&\overline{c_k}\\c_k&0\end{pmatrix},
$$
whose eigenvalues are $\pm|c_k|$. Conversely, \eqref{equation:sharp-character-action} says precisely that $B$ is a weighted partial permutation matrix, so the labeling is sharp.
\end{proof}

\begin{remark}
To construct a simple graph from the normal form, one chooses data satisfying \eqref{equation:sharp-hermitian-data} and \eqref{equation:sharp-real-data} and imposes $M(x,x)=0$ and $M(x,y)\in\{0,1\}$ for $x\ne y$. These conditions require the kernel in \eqref{equation:sharp-inverse-normal-form} to be the adjacency matrix of a simple graph, and together with the symmetry conditions they parametrize all sharp labeled simple graphs.
\end{remark}

\begin{corollary}[Even walks and a spectral obstruction, cf.\ \cite{GuptaIosevichEtAl2026}]
\label{corollary:sharp-even-walks}
Under a sharp labeling, $M^{2r}$ is circulant for every integer $r\geq1$. In particular the number of closed walks of each even length is independent of the vertex, and the number of common neighbors of $x$ and $y$ depends only on $x-y$. If the adjacency spectrum has no pair of nonzero eigenvalues $\lambda,-\lambda$, then every sharp labeling has a circulant adjacency matrix.
\end{corollary}

\begin{proof}
By the normal form, $B^2_{k,k}=|c_k|^2$ on $\Omega$, with all other entries zero, so $B^{2r}$ is diagonal and $M^{2r}$ is circulant, and the entries of $M^2$ count common neighbors. Finally, a nontrivial two-element orbit of $\tau$ would contribute opposite nonzero eigenvalues, so under the stated spectral hypothesis $B$ is diagonal.
\end{proof}

\begin{remark}
The circulant property of $M^2$ is a useful necessary condition, but it leaves some freedom inside a repeated eigenspace of $M^2$, where a Hermitian square root can mix Fourier characters, whereas sharpness requires each individual character to be sent to a scalar multiple of one character.
\end{remark}

The entries of an adjacency matrix are integers, so applying a field automorphism to its Fourier transform only permutes the roots of unity, and this observation puts an arithmetic restriction on the pairing $\tau$. For a unit $a\in{\mathbb Z}_N^\times$, let $\sigma_a$ be the automorphism of the cyclotomic field ${\mathbb Q}(\zeta)$ determined by $\sigma_a(\zeta)=\zeta^a$.

\begin{proposition}[Arithmetic restrictions on the support involution]
\label{proposition:sharp-galois}
For every sharp labeling and every $a\in{\mathbb Z}_N^\times$,
\begin{equation}
a\Omega=\Omega,\tau(ak)=a\tau(k),c_{ak}=\sigma_a(c_k).
\label{equation:sharp-galois}
\end{equation}
If $N$ is odd, then $\tau$ preserves the additive order of every active frequency, where the additive order of $k$ is the least positive integer $q$ such that $qk=0$, and for every divisor $q$ of $N$ whose order-$q$ frequencies are active there is a unit $a_q$ modulo $q$ such that
\begin{equation}
a_q^2\equiv1\pmod q
\text{ and }
\tau(k)=a_qk\text{ when $k$ has additive order $q$}.
\label{equation:sharp-divisor-multiplier}
\end{equation}
Thus each odd prime-power factor of $q$ independently has multiplier $+1$ or $-1$ in this involution.
\end{proposition}

\begin{proof}
The formula
$$
B_{u,v}=\frac{1}{N}\sum_{x,y}M(x,y)\zeta^{-ux+vy}
$$
shows that $\sigma_a(B_{u,v})=B_{au,av}$, because the entries $M(x,y)$ are integers. Galois automorphisms preserve nonzero entries, so the uniqueness of the occupied entry in each active column proves \eqref{equation:sharp-galois}.

Suppose now that $N$ is odd. A frequency of additive order $q$ has stabilizer
$$
K_q=\{a\in{\mathbb Z}_N^\times:a\equiv1\pmod q\}
$$
under multiplication by units, and the subgroups $K_q$ distinguish the divisors of odd $N$. Indeed, by the Chinese remainder theorem this reduces to the subgroups of ${\mathbb Z}_{p^e}^\times$ defined by $a\equiv1\pmod{p^j}$ for $0\leq j\leq e$, and these are strictly nested when $p$ is odd, since the first step excludes a unit not congruent to one modulo $p$ and each subsequent step excludes $1+p^j$. An equivariant involution preserves stabilizers in both directions, so $k$ and $\tau(k)$ have the same additive order.

The units act transitively on the frequencies of each fixed order $q$. If $k_0=\frac{N}{q}$, write $\tau(k_0)=a_qk_0$, and equivariance then forces $\tau(k)=a_qk$ on the entire order-$q$ orbit, while involutivity yields $a_q^2\equiv1\pmod q$. For an odd prime power $p^j$, the factors $a_q-1$ and $a_q+1$ have greatest common divisor dividing two, so $p^j$ divides one of them, and the final assertion follows from the Chinese remainder theorem.
\end{proof}

\begin{theorem}[Complete classification at odd prime-power order]
\label{theorem:sharp-prime-power}
Let $N=p^r$ with $p$ an odd prime and $r\geq1$. A simple graph with at least one edge on $N$ vertices attains the graph-energy bound for some labeling if and only if it is a circulant graph, and in fact every sharp labeling is itself circulant.
\end{theorem}

\begin{proof}
By Proposition~\ref{proposition:sharp-galois}, on each active frequency orbit of order $p^j$ the involution is either $k\mapsto k$ or $k\mapsto-k$. Let $\Omega_-$ be the union of the nonzero orbits on which $\tau(k)=-k\ne k$. The zero diagonal and \eqref{equation:sharp-inverse-normal-form} give, for every $x$,
$$
0=M(x,x)=\frac{1}{N}\left(\sum_{k\in\Omega\setminus\Omega_-}c_k+\sum_{k\in\Omega_-}c_k\zeta^{-2kx}\right).
$$
Because $N$ is odd, the nonzero frequencies $-2k$ in the second sum are distinct, and the linear independence of the cyclic characters forces $c_k=0$ for every $k\in\Omega_-$, contrary to activity. Hence $\Omega_-$ is empty, $B$ is diagonal, and $M$ is circulant. Conversely, every circulant adjacency matrix is diagonalized by $U$ and consequently has a sharp labeling.
\end{proof}

At an order with more than one odd prime factor, the involution can use different signs on different factors, and the following construction from \cite[Theorem 2.16]{GuptaIosevichEtAl2026} turns this possibility into sharp graphs.

\begin{proposition}[Twisted difference construction]
\label{proposition:sharp-twisted-difference}
Let $a\in{\mathbb Z}_N^\times$ satisfy $a^2=1$, and choose a nonempty $S\subseteq{\mathbb Z}_N$ satisfying
$$
S=-aS
\text{ and }
S\cap(1-a){\mathbb Z}_N=\varnothing.
$$
Then
\begin{equation}
M(x,y)=1_S(y-ax)
\label{equation:sharp-twisted-graph}
\end{equation}
is the adjacency matrix of a simple $|S|$-regular graph with a sharp labeling, whose Fourier action is
$$
M\chi_k=\left(\sum_{s\in S}\zeta^{ks}\right)\chi_{ak}.
$$
\end{proposition}

\begin{proof}
The second condition makes the diagonal zero, and since $x-ay=-a(y-ax)$, the first condition supplies symmetry. Each row has $|S|$ ones. Substituting $y=ax+s$ directly produces the character action, which is a weighted partial permutation because multiplication by $a$ permutes the frequencies.
\end{proof}

\begin{example}[A connected nonsingular sharp graph that is not Cayley]
\label{example:sharp-fifteen}
Identify ${\mathbb Z}_{15}$ with ${\mathbb Z}_3\times{\mathbb Z}_5$ and join $(u,v)$ and $(u',v')$ exactly when
\begin{equation}
u'\ne u
\text{ and }
v+v'\in\{0,1\}\pmod5.
\label{equation:sharp-fifteen-edges}
\end{equation}
This is the twisted construction with $a=4\pmod{15}$, which acts as $(u,v)\mapsto(u,-v)$, and $S=\{1,2\}\times\{0,1\}$, so the graph is simple, $4$-regular and sharp.

Let us examine this graph directly. Write $J_3$ for the $3\times3$ all-ones matrix, put $K=J_3-I_3$, and let $H_{v,w}=1_{\{0,1\}}(v+w)$ on ${\mathbb Z}_5$, so that the adjacency matrix in product coordinates is $K\otimes H$. The matrix $H$ is symmetric and has two ones in each row, and its edges between distinct vertices form the path
$$
0\longleftrightarrow1\longleftrightarrow4\longleftrightarrow2\longleftrightarrow3,
$$
with a loop at each endpoint, as in Figure~\ref{figure:fifteen}. Its eigenvalue $2$ is simple, since if $Hz=2z$ and we choose a coordinate of maximal modulus, equality in the averaging relation propagates its value along the path, so that $z$ is constant, and all eigenvalues of $H$ have modulus at most $2$. Since the eigenvalues of $K$ are $2,-1,-1$, the eigenvalue $4$ of $K\otimes H$ is simple, because it can only arise from the product $2\cdot2$. A $4$-regular graph has one independent $4$-eigenvector for each connected component, so the graph is connected.

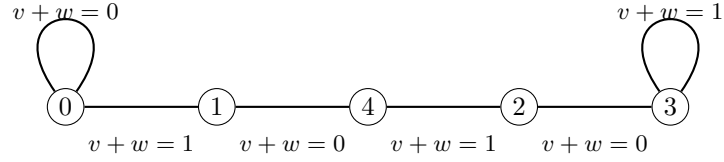
\begin{figure}[htb]
\centering
\begin{tikzpicture}[x=1cm,y=1cm,every node/.style={circle,draw,inner sep=2pt,font=\small}]
\node (a) at (0,0) {$0$};
\node (b) at (2,0) {$1$};
\node (c) at (4,0) {$4$};
\node (d) at (6,0) {$2$};
\node (e) at (8,0) {$3$};
\draw[thick] (a) -- (b) -- (c) -- (d) -- (e);
\draw[thick] (a) to[out=125,in=55,looseness=14] (a);
\draw[thick] (e) to[out=125,in=55,looseness=14] (e);
\node[draw=none,font=\footnotesize] at (0,1.25) {$v+w=0$};
\node[draw=none,font=\footnotesize] at (8,1.25) {$v+w=1$};
\node[draw=none,font=\footnotesize] at (1,-0.5) {$v+w=1$};
\node[draw=none,font=\footnotesize] at (3,-0.5) {$v+w=0$};
\node[draw=none,font=\footnotesize] at (5,-0.5) {$v+w=1$};
\node[draw=none,font=\footnotesize] at (7,-0.5) {$v+w=0$};
\end{tikzpicture}
\caption{The factor $H_{v,w}=1_{\{0,1\}}(v+w)$ on ${\mathbb Z}_5$ in Example~\ref{example:sharp-fifteen}, drawn as a graph with loops. Each row has exactly two ones, and the adjacency matrix of the $15$-vertex example is $(J_3-I_3)\otimes H$, so a vertex $(u,v)$ is joined to the two vertices $(u',w)$ with $u'\ne u$ and $w$ a neighbor of $v$ in this picture, loops included.}
\label{figure:fifteen}
\end{figure}

The character action of $H$ has coefficients $1+e^{\frac{2\pi ik}{5}}$, none of which vanishes, so $H$, and consequently $K\otimes H$, is nonsingular.

Nevertheless the graph is not vertex-transitive. The number of triangles containing $(u,v)$ is
$$
\frac{(K^3)_{u,u}(H^3)_{v,v}}{2}=(H^3)_{v,v},
$$
and $(H^3)_{v,v}$ counts the triples $s_1,s_2,s_3\in\{0,1\}$ with $s_1-s_2+s_3=2v\pmod5$, which at $v=0,1,2,3,4$ number respectively $3,1,1,3,0$. An automorphism preserves the number of triangles containing a vertex, and since these counts differ, the graph is not vertex-transitive. Every Cayley graph is vertex-transitive under left translations, so this sharp graph is not a Cayley graph of any group, and in particular connectedness and nonsingularity do not extend Theorem~\ref{theorem:sharp-prime-power} to arbitrary odd orders. The order $15$ is the smallest possible odd order for a sharp graph that is not vertex-transitive, because every odd integer $N$ with $1<N<15$ is a prime power, so that Theorem~\ref{theorem:sharp-prime-power} makes every sharp graph of those orders circulant and hence vertex-transitive.
\end{example}

\begin{proposition}[A labeling and its complement]
\label{proposition:sharp-complement}
Suppose that $G$ and its complement both have at least one edge. A fixed labeling is sharp for both graphs if and only if it is circulant.
\end{proposition}

\begin{proof}
If $B=UMU^*$ is a weighted partial permutation matrix, regularity forces $B_{0,0}=d$ and no other occupied entry in row or column zero. For the complement the conjugated adjacency matrix is
$$
B^{\mathrm c}=NE_{0,0}-I-B,
$$
where $E_{0,0}$ has a single one in position $(0,0)$ and zeros elsewhere. Any off-diagonal occupied entry of $B$ occurs in a nonzero row whose diagonal entry is zero, and that row of $B^{\mathrm c}$ then contains both the negative of the off-diagonal entry and a diagonal entry $-1$, contradicting sharpness. Thus $B$ must be diagonal. Conversely, if $B$ is diagonal then so is $B^{\mathrm c}$, and both labelings are sharp.
\end{proof}

For general $N$ the classification reduces to solving the zero-diagonal and zero-one conditions in the Fourier involution normal form, subject to the arithmetic restrictions above. The examples above realize some of these data through the twisted difference construction, and a classification at general order would require a description of all data satisfying those conditions.

\section{Sparse Fourier approximation of adjacency functions}
\label{section:approximation}

How many Fourier coefficients must we retain to approximate an adjacency function? Orthogonality tells us to keep the largest coefficients, and we begin with this observation on a finite abelian group.

Let $H$ be a finite abelian group of order $M$. A trigonometric polynomial on $H$ is a finite linear combination of characters, and the characters that occur with nonzero coefficients are called its Fourier frequencies. For an integer $k\geq0$, let $\mathcal T_k(H)$ be the set of functions of the form
$$
P(x)=\sum_{\gamma\in S}c_\gamma\gamma(x),
$$
where $S\subset\widehat H$, $|S|\leq k$ and $c_\gamma\in{\mathbb C}$, so that $P$ uses at most $k$ Fourier frequencies. In particular $\mathcal T_0(H)=\{0\}$, and since $|\widehat H|=M$, one has $\mathcal T_k(H)=\mathcal T_M(H)$ whenever $k\geq M$. The best $k$-term Fourier approximation error is
$$
\sigma_k(f)_2=\inf_{P\in\mathcal T_k(H)}\|f-P\|_2.
$$

\begin{proposition}[Exact best $k$-term approximation]
\label{proposition:exact-best-k-general}
Let $0\leq k\leq M$, and let $a_1\geq a_2\geq\cdots\geq a_M$ be the $M$ numbers $|\widehat f(\gamma)|$, $\gamma\in\widehat H$, listed in nonincreasing order with repetitions retained. Then
$$
\sigma_k(f)_2^2=\sum_{j>k}a_j^2,
$$
and in particular the best $k$-term approximation is obtained by retaining the $k$ largest Fourier coefficients.
\end{proposition}

\begin{proof}
Fix $S\subset\widehat H$ with $|S|\leq k$. The functions whose Fourier transforms are supported in $S$ form a linear subspace of $\ell^2(H)$, and the orthogonal projection of $f$ onto this subspace has Fourier transform equal to $\widehat f$ on $S$ and to zero off $S$. By Plancherel,
$$
\|f-P\|_2^2=\sum_{\gamma\notin S}|\widehat f(\gamma)|^2,
$$
and this sum is minimized by choosing $S$ to contain the $k$ largest values of $|\widehat f(\gamma)|$, so the minimum is $\sum_{j>k}a_j^2$.
\end{proof}

\begin{definition}
For a nonzero function $f:H\to{\mathbb C}$ and $0<\eta<1$, its Fourier approximation complexity is
$$
\kappa_\eta(f)=\min\{k\in\{0,1,\ldots,M\}:\sigma_k(f)_2\leq\eta\|f\|_2\},
$$
a minimum that exists because $\sigma_M(f)_2=0$. For a labeled graph $(G,\pi)$ we write $\kappa_\eta(G,\pi)=\kappa_\eta(f_\pi)$, so that $\kappa_\eta(G,\pi)$ counts the Fourier frequencies needed to approximate the adjacency function to relative $L^2$ error at most $\eta$.
\end{definition}

\begin{theorem}[Spectral lower bound and sharp labelings]
\label{theorem:spectral-best-k}
Let $G$ be a simple graph with at least one edge and adjacency matrix $A$, and list its singular values as $a_1\geq\cdots\geq a_N\geq0$. For every labeling $\pi$ and every integer $0\leq k\leq N$,
$$
\sigma_k(f_\pi)_2^2\geq\sum_{j>k}a_j^2.
$$
If $\pi_*$ is a sharp labeling, equality holds for every $k$. Consequently, the same labeling minimizes the best $k$-term error for every $k$ and the relative approximation complexity for every $0<\eta<1$, and more precisely
$$
\min_\pi\sigma_k(f_\pi)_2^2=\sum_{j>k}a_j^2
\text{ and }
\min_\pi\kappa_\eta(G,\pi)=\min\left\{k:\sum_{j>k}a_j^2\leq\eta^2\sum_{j=1}^Na_j^2\right\},
$$
where the first identity is understood with zero tails when $k\geq N$.
\end{theorem}

\begin{proof}
Each character $(x,y)\mapsto e^{\frac{2\pi i(mx+ny)}{N}}$ is a rank-one matrix, so a $k$-term Fourier polynomial $P$ has matrix rank at most $k$. Let $R$ be the orthogonal projector onto its column space. Since $(I-R)P=0$, we have
$$
\|A_\pi-P\|_{\mathrm F}^2\geq\|(I-R)(A_\pi-P)\|_{\mathrm F}^2=\|(I-R)A_\pi\|_{\mathrm F}^2=\|A\|_{\mathrm F}^2-\operatorname{tr}(RA_\pi A_\pi^*).
$$
If $u_j$ are orthonormal eigenvectors of $A_\pi A_\pi^*$ with eigenvalues $a_j^2$, then the numbers $w_j=\langle Ru_j,u_j\rangle$ lie in $[0,1]$ and sum to at most $k$, and therefore
$$
\operatorname{tr}(RA_\pi A_\pi^*)=\sum_ja_j^2w_j\leq\sum_{j\leq k}a_j^2,
$$
which proves the lower bound after taking the infimum over $P$. For a sharp labeling the nonzero absolute Fourier coefficients are exactly the nonzero singular values of $A$, so retaining the $k$ largest Fourier coefficients attains equality, and the assertions about minimization follow.
\end{proof}

\begin{remark}
The lower bound is the usual best rank-$k$ approximation bound, which applies here because a sum of $k$ Fourier characters has rank at most $k$. The Singer graphs below have sharp labelings, so their exact approximation formulas also give the minimum over all labelings.
\end{remark}

\begin{theorem}[Collision obstruction to sparse approximation]
\label{theorem:collision-approximation-general}
Let $f:H\to{\mathbb C}$ be nonzero. Then
\begin{equation}
\frac{\sigma_k(f)_2^2}{\|f\|_2^2}\geq\max\left\{0,1-\sqrt{\frac{k}{D_4(f)}}\right\}.
\label{equation:collision-approximation-lower}
\end{equation}
Consequently, if $P\in\mathcal T_k(H)$ and $\|f-P\|_2<\eta\|f\|_2$ for some $0<\eta<1$, then
\begin{equation}
k>(1-\eta^2)^2D_4(f).
\label{equation:necessary-k-general}
\end{equation}
\end{theorem}

\begin{proof}
For any frequency set $S$ of size at most $k$, the Cauchy--Schwarz inequality shows that
$$
\sum_{\gamma\in S}|\widehat f(\gamma)|^2\leq\sqrt k\|\widehat f\|_4^2=\sqrt{\frac{k}{D_4(f)}}\|f\|_2^2.
$$
By Plancherel, an approximation supported on $S$ leaves at least the energy outside $S$, and taking the infimum yields \eqref{equation:collision-approximation-lower}. If the relative error is strictly less than $\eta$, then $1-\sqrt{\frac{k}{D_4(f)}}<\eta^2$, which rearranges to the second assertion.
\end{proof}

The collision dimension tells us how many coefficients are necessary. To obtain an upper bound we sample Fourier coefficients with probabilities proportional to their absolute values, which is the empirical approximation method of Maurey in the form used by Barron \cite{Barron1993}, and related formulations in terms of the Fourier ratio appear in \cite{AldalehEtAl2025}.

\begin{theorem}[Fourier-ratio upper bound]
\label{theorem:fr-approximation-upper-general}
Let $f:H\to{\mathbb C}$ be nonzero, let $0<\eta<1$, and let $k$ be an integer satisfying
$$
k>\frac{\operatorname{FR}(f)^2-1}{\eta^2}.
$$
Then there exists $P\in\mathcal T_k(H)$ such that $\|f-P\|_2<\eta\|f\|_2$.
\end{theorem}

\begin{proof}
Write $L=\|\widehat f\|_1$, choose a random frequency $\Gamma$ with probability $\frac{|\widehat f(\gamma)|}{L}$ of equalling $\gamma$, and set
$$
Y=L\frac{\widehat f(\Gamma)}{|\widehat f(\Gamma)|}1_{\{\Gamma\}}.
$$
Then ${\mathbb E}Y=\widehat f$, $\|Y\|_2=L$, and ${\mathbb E}\|Y-\widehat f\|_2^2=L^2-\|f\|_2^2$. For independent copies $Y_1,\ldots,Y_k$, their average $Q$ has support of size at most $k$ and
$$
{\mathbb E}\|Q-\widehat f\|_2^2=\frac{L^2-\|f\|_2^2}{k}=\frac{\operatorname{FR}(f)^2-1}{k}\|f\|_2^2,
$$
which the hypothesis makes smaller than $\eta^2\|f\|_2^2$. Some realization attains an error no larger than this average, and Fourier inversion supplies the required $P$.
\end{proof}

\begin{corollary}[Two-sided approximation complexity]
\label{corollary:two-sided-general}
For every nonzero $f:H\to{\mathbb C}$ and every $0<\eta<1$,
\begin{equation}
\lceil(1-\eta^2)^2D_4(f)\rceil\leq\kappa_\eta(f)\leq\min\left\{|\operatorname{supp}(\widehat f)|,\left\lfloor\frac{\operatorname{FR}(f)^2-1}{\eta^2}\right\rfloor+1\right\}.
\label{equation:two-sided-kappa-general}
\end{equation}
\end{corollary}

\begin{proof}
The collision lower bound forces $k\geq(1-\eta^2)^2D_4(f)$ whenever the relative error is at most $\eta$. For the upper bound apply Theorem~\ref{theorem:fr-approximation-upper-general} with $k=\lfloor\frac{\operatorname{FR}(f)^2-1}{\eta^2}\rfloor+1$, while retaining all nonzero coefficients accounts for the other upper bound.
\end{proof}

For adjacency functions, the energy identity turns these estimates into the following bounds.

\begin{theorem}[Two-sided graph approximation bounds]
\label{theorem:two-sided-graph-approximation}
Let $G$ have $N$ vertices and $s>0$ edges, and let $\pi$ be a labeling. Then, for every $0<\eta<1$,
\begin{equation}
\left\lceil(1-\eta^2)^2\frac{N^2(2s)^2}{E_{\mathrm{edge}}(G,\pi)}\right\rceil\leq\kappa_\eta(G,\pi)\leq\min\left\{|\operatorname{supp}(\widehat f_\pi)|,\left\lfloor\frac{\operatorname{FR}(G,\pi)^2-1}{\eta^2}\right\rfloor+1\right\}.
\label{equation:two-sided-graph-approximation}
\end{equation}
\end{theorem}

\begin{proof}
Apply Corollary \ref{corollary:two-sided-general} to $f_\pi$ and use the formula \eqref{equation:graph-collision-dimension}.
\end{proof}

\begin{remark}
Theorem \ref{theorem:two-sided-graph-approximation} bounds the number of Fourier frequencies needed to represent an approximation. A lower bound for the number of observed adjacency entries would require a separate argument about the observation model.
\end{remark}

Let us apply these bounds to the Singer Levi graphs from \cite{GuptaIosevichEtAl2026}. These graphs have sharp labelings, so their adjacency spectra determine both their additive energy and their best Fourier approximation errors.

\begin{corollary}[Singer graphs have linear minimum Fourier model size]
\label{corollary:singer-linear-complexity}
Let $G_q$ be the Singer Levi graph constructed in \cite{GuptaIosevichEtAl2026}, where $q\geq2$ is a prime power. Its two vertex classes are the points and the lines of the projective plane of order $q$, a point being adjacent to a line when it lies on that line, and the construction in \cite{GuptaIosevichEtAl2026} uses a Singer difference set. Put $v=q^2+q+1$ and $N_q=2v$. The graph is $(q+1)$-regular, and its adjacency spectrum consists of $q+1$ and $-(q+1)$ with multiplicity one each and of $\sqrt q$ and $-\sqrt q$ with multiplicity $v-1$ each. Let $\pi_q$ denote the sharp labeling constructed there. Then
\begin{equation}
D_4(G_q,\pi_q)=\frac{4v^2(q+1)^2}{2(q+1)^4+2(v-1)q^2}\asymp q^2\asymp N_q
\label{equation:singer-d4}
\end{equation}
and
\begin{equation}
\operatorname{FR}(G_q,\pi_q)^2=\frac{(2(q+1)+2(v-1)\sqrt q)^2}{2v(q+1)}\asymp q^2\asymp N_q,
\label{equation:singer-fr}
\end{equation}
and consequently, for every fixed $0<\eta<1$,
$$
\min_\pi\kappa_\eta(G_q,\pi)=\kappa_\eta(G_q,\pi_q)\asymp_\eta N_q.
$$
\end{corollary}

\begin{proof}
The spectrum and the existence of the sharp labeling are proved in \cite[Theorem 2.25 and Proposition 2.26]{GuptaIosevichEtAl2026}. Since $2s=N_q(q+1)$, Theorem \ref{theorem:extremizer-exact-energy} yields
$$
D_4(G_q,\pi_q)=\frac{N_q^2(q+1)^2}{2(q+1)^4+2(v-1)q^2},
$$
which is \eqref{equation:singer-d4}, and substituting the eigenvalues into the graph-energy formula proves \eqref{equation:singer-fr}. Both quantities are comparable to $q^2$, so the estimate for $\kappa_\eta(G_q,\pi_q)$ follows from Theorem~\ref{theorem:two-sided-graph-approximation}, and the minimization over labelings follows from Theorem~\ref{theorem:spectral-best-k}.
\end{proof}

\begin{remark}[A nonabelian Cayley realization]
The Singer Levi graphs are not abelian Cayley graphs \cite{GuptaIosevichEtAl2026}, but they do admit a Cayley realization on a nonabelian group. Write the Singer Levi graph as $P_i\sim L_j$ if and only if $i+j\in D$, where $D\subset{\mathbb Z}_v$ is the Singer difference set used in \cite{GuptaIosevichEtAl2026}. The maps $r$ and $t$ defined by
$$
r(P_i)=P_{i+1},r(L_j)=L_{j-1},t(P_i)=L_i,t(L_i)=P_i
$$
preserve adjacency and satisfy $r^v=t^2=1$ and $trt=r^{-1}$. The resulting dihedral group acts transitively and freely on the $2v$ vertices, since the images of $P_0$ under $r^i$ and $tr^i$ run through the distinct vertices $P_i$ and $L_i$, and identifying vertices with these group elements therefore exhibits a Cayley realization on a nonabelian group.
\end{remark}

\begin{remark}[Exact approximation in the Singer labeling]
We can also compute the error exactly. The squared absolute Fourier coefficients are $(q+1)^2$ twice and $q$ exactly $2(v-1)$ times, and therefore
$$
\sigma_k(f_{\pi_q})_2^2=
\begin{cases}
(2-k)(q+1)^2+2(v-1)q,&0\leq k\leq2,\\
(2v-k)q,&2\leq k\leq2v,\\
0,&k\geq2v,
\end{cases}
$$
where the two expressions agree at $k=2$. In particular, for each fixed $0<\eta<1$ and all sufficiently large $q$,
$$
\kappa_\eta(f_{\pi_q})=\left\lceil2v-\eta^2\frac{2v(q+1)}{q}\right\rceil
\text{ and }
\frac{\kappa_\eta(f_{\pi_q})}{N_q}\longrightarrow1-\eta^2,
$$
which follows directly by retaining the largest coefficients.
\end{remark}

\section{Centered adjacency and nonconstant structure}
\label{section:centered}

For a dense graph a good approximation may tell us little more than its density, since the zero Fourier coefficient can account for most of the second and fourth moments. To see what remains, we subtract the mean, and the following calculation applies to any subset of a finite abelian group.

\begin{proposition}[Centered energy identity]
\label{proposition:centered-energy-general}
Let $A$ be a nonempty proper subset of a finite abelian group $H$ of order $M$, and set $\alpha=\frac{|A|}{M}$ and $g_A=1_A-\alpha$. Then
\begin{equation}
\|g_A\|_2^2=|A|(1-\alpha)
\label{equation:centered-l2-general}
\end{equation}
and
\begin{equation}
\mathcal E_+(g_A)=E_+(A)-\frac{|A|^4}{M},
\label{equation:centered-energy-general}
\end{equation}
and consequently
\begin{equation}
D_4(g_A)=\frac{M|A|^2(1-\alpha)^2}{E_+(A)-\frac{|A|^4}{M}}.
\label{equation:centered-d4-general}
\end{equation}
\end{proposition}

\begin{proof}
The centered function has squared norm $|A|(1-\alpha)^2+(M-|A|)\alpha^2=|A|(1-\alpha)$. Its transform agrees with $\widehat{1_A}$ away from the trivial character, where it vanishes, so centering removes exactly $M(\frac{|A|}{\sqrt M})^4=\frac{|A|^4}{M}$ from the additive energy. Since $A$ is nonempty and proper, $g_A\ne0$ and its weighted energy is positive, and the formula for $D_4$ follows from the energy identity.
\end{proof}

For a simple graph with at least one edge, the directed edge set is always a proper subset of ${\mathbb Z}_N^2$, because it contains no diagonal points, so centering is available for every such graph.

\begin{definition}
Let $G$ have $N$ vertices and $s>0$ edges, and let $\pi$ be a labeling. Define the edge density $\alpha_G=\frac{2s}{N^2}$ and the centered adjacency function $g_\pi=f_\pi-\alpha_G$, and define the centered Fourier ratio and the centered collision dimension by
$$
\operatorname{FR}^{\circ}(G,\pi)=\operatorname{FR}(g_\pi)
\text{ and }
D_4^{\circ}(G,\pi)=D_4(g_\pi).
$$
\end{definition}

\begin{corollary}[Centered graph formula]
\label{corollary:centered-graph-formula}
Let $G$ have $N$ vertices and $s>0$ edges. Then
\begin{equation}
D_4^{\circ}(G,\pi)=\frac{N^2(2s)^2(1-\frac{2s}{N^2})^2}{E_{\mathrm{edge}}(G,\pi)-\frac{(2s)^4}{N^2}}
\label{equation:centered-graph-d4}
\end{equation}
and $D_4^{\circ}(G,\pi)\leq\operatorname{FR}^{\circ}(G,\pi)^2$. If $0<\eta<1$ and a trigonometric polynomial $P$ with at most $k$ frequencies satisfies $\|g_\pi-P\|_2<\eta\|g_\pi\|_2$, then $k>(1-\eta^2)^2D_4^{\circ}(G,\pi)$.
\end{corollary}

\begin{proof}
Apply Proposition~\ref{proposition:centered-energy-general} with $H={\mathbb Z}_N^2$ and $|A_\pi|=2s$. The collision hierarchy and Theorem~\ref{theorem:collision-approximation-general} give the remaining claims.
\end{proof}

For a sharp labeling, subtracting the mean removes one adjacency eigenvalue from the formulas. We write $J$ for the $N\times N$ matrix all of whose entries are one, and as before we identify a matrix indexed by ${\mathbb Z}_N^2$ with the function given by its entries.

\begin{theorem}[Centered formulas for sharp labelings]
\label{theorem:centered-extremizer-formulas}
Let $\pi$ be a sharp labeling of $G$, which is regular by Theorem \ref{theorem:prior-extremizer-regularity}. Write $d>0$ for its degree and $A$ for the corresponding adjacency matrix. Since $G$ is $d$-regular, $2s=Nd$ and the density $\frac{2s}{N^2}$ equals $\frac{d}{N}$, so the centered adjacency function defined earlier is
$$
g_\pi=A-\frac{d}{N}J,
$$
regarded as a function of the two matrix indices. Then
\begin{equation}
\mathcal E_+(g_\pi)=N^2(\operatorname{tr}(A^4)-d^4),
\label{equation:centered-extremizer-energy}
\end{equation}
\begin{equation}
D_4^{\circ}(G,\pi)=\frac{(Nd-d^2)^2}{\operatorname{tr}(A^4)-d^4},
\label{equation:centered-extremizer-d4}
\end{equation}
and
\begin{equation}
\operatorname{FR}^{\circ}(G,\pi)=\frac{\mathcal E(G)-d}{\sqrt{Nd-d^2}}.
\label{equation:centered-extremizer-fr}
\end{equation}
The absolute values of the nonzero Fourier coefficients of $g_\pi$ are the nonzero adjacency singular values with one copy of $d$ removed.
\end{theorem}

\begin{proof}
Regularity means that $\widehat A(0,0)=d$, and centering removes exactly that coefficient. By Theorem~\ref{theorem:extremizer-exact-energy} the other nonzero coefficient moduli are the nonzero singular values with one copy of $d$ removed, hence
$$
\|\widehat g_\pi\|_1=\mathcal E(G)-d,\|\widehat g_\pi\|_2^2=Nd-d^2\text{ and }\|\widehat g_\pi\|_4^4=\operatorname{tr}(A^4)-d^4,
$$
and the definitions together with the energy identity prove the assertions.
\end{proof}

Let us see what these formulas say for two familiar dense graphs.

\begin{example}[The complete graph]
\label{example:complete-graph-centered}
Let $G=K_N$ with $N\geq2$, and use any labeling. Its adjacency function is $f(x,y)=1-1_{\{x=y\}}$, and a direct calculation shows that
\begin{equation}
\widehat f(m,n)=N1_{\{(m,n)=(0,0)\}}-1_{\{m+n=0\}}.
\label{equation:complete-graph-fourier}
\end{equation}
Indeed, the constant function one has Fourier transform $N$ at $(0,0)$ and zero elsewhere, while
$$
\frac{1}{N}\sum_{x\in{\mathbb Z}_N}e^{\frac{-2\pi i(m+n)x}{N}}=1_{\{m+n=0\}}.
$$
Thus $\widehat f(0,0)=N-1$, and at each of the other $N-1$ points on the line $m+n=0$ the Fourier coefficient is $-1$. Therefore $\|\widehat f\|_1=2(N-1)$ and $\|\widehat f\|_2^2=(N-1)^2+(N-1)=N(N-1)$, so that
$$
\operatorname{FR}(f)^2=4\left(1-\frac{1}{N}\right),
$$
and moreover $\|\widehat f\|_4^4=(N-1)^4+(N-1)$, so that
$$
D_4(f)=\frac{N^2(N-1)^2}{(N-1)^4+(N-1)}=\frac{N(N-1)}{N^2-3N+3}\longrightarrow1
$$
as $N\longrightarrow\infty$.

Why is the uncentered collision dimension so small? The function $f$ is close to its mean. The edge density is $\alpha=1-\frac{1}{N}$, the orthogonal projection of $f$ onto the constant functions is the constant function $\alpha$, and by Plancherel, or by direct calculation, $\|f-\alpha\|_2^2=N-1$ while $\|f\|_2^2=N(N-1)$, so the best constant approximation has relative error
$$
\frac{\|f-\alpha\|_2}{\|f\|_2}=\frac{1}{\sqrt N}.
$$
The centered adjacency function is
$$
g(x,y)=f(x,y)-\left(1-\frac{1}{N}\right)=\frac{1}{N}-1_{\{x=y\}}.
$$
Its zero Fourier coefficient vanishes, at every nonzero point on the line $m+n=0$ its Fourier coefficient is $-1$, and all other coefficients vanish, so there are exactly $N-1$ nonzero coefficients, all of modulus one. Therefore
$$
D_4(g)=\operatorname{FR}(g)^2=N-1
\text{ and }
\kappa_\eta(g)=\lceil(1-\eta^2)(N-1)\rceil,
$$
and approximating the centered adjacency function to any fixed relative error less than one requires a number of Fourier frequencies proportional to $N$, whereas the uncentered function has a constant approximation whose relative error tends to zero.
\end{example}

\begin{example}[Parity-labeled complete bipartite graphs]
\label{example:complete-bipartite-centered}
Suppose that $N$ is even. Partition ${\mathbb Z}_N$ into the even and odd residue classes and consider the complete bipartite graph $K_{\frac{N}{2},\frac{N}{2}}$ with this labeling. Its adjacency function is
$$
f(x,y)=1_{\{x-y\text{ is odd}\}}=\frac{1-(-1)^{x-y}}{2},
$$
whose Fourier transform is supported at two points, with $\widehat f(0,0)=\frac{N}{2}$ and $\widehat f(\frac{N}{2},\frac{N}{2})=-\frac{N}{2}$, so that $D_4(f)=\operatorname{FR}(f)^2=2$. The edge density is $\frac{1}{2}$, so the centered adjacency function is
$$
g(x,y)=f(x,y)-\frac{1}{2}=-\frac{1}{2}(-1)^{x-y},
$$
a scalar multiple of a single character on ${\mathbb Z}_N^2$, and consequently $D_4(g)=\operatorname{FR}(g)^2=1$. Thus the nonconstant part has one Fourier frequency.
\end{example}

\begin{remark}
These examples explain why subtracting the mean matters. For the complete graph, most of the $L^2$ mass lies in the constant term while the nonconstant part has $N-1$ Fourier frequencies of equal size, whereas for the parity-labeled complete bipartite graph the entire nonconstant part is a scalar multiple of one character.
\end{remark}

\section{Additive structure, separation, and random labelings}
\label{section:inverse}

A small Fourier ratio forces large additive energy, and the Balog--Szemer\'edi--Gowers theorem then finds a large subset of the directed edge set whose difference set is small. For a finite subset $B$ of an abelian group, its difference set is
$$
B-B=\{b-b':b,b'\in B\},
$$
and we measure the additive structure of $B$ by comparing $|B-B|$ with $|B|$. The quantitative form of the Balog--Szemer\'edi--Gowers theorem that we use is due to Reiher and Schoen \cite{ReiherSchoen2024}, and it is recorded in Appendix \ref{appendix:bsg}.

\begin{theorem}[Low Fourier ratio forces structured edges]
\label{theorem:graph-bsg}
Let $G$ be a simple graph with $N$ vertices and $s>0$ edges, let $\pi$ be a labeling, and suppose that $\operatorname{FR}(G,\pi)\leq r$. Set $\alpha=\frac{2s}{N^2}$ and $K=\alpha r^2$. Then $K\geq1$ and
\begin{equation}
E_{\mathrm{edge}}(G,\pi)\geq\frac{(2s)^3}{K},
\label{equation:graph-large-energy-from-fr}
\end{equation}
and for every $0<\varepsilon<\frac{1}{2}$ there exists a subset $A'\subset A_\pi$ such that
\begin{equation}
|A'|\geq(1-\varepsilon)K^{-\frac{1}{2}}(2s)
\label{equation:graph-bsg-size}
\end{equation}
and
\begin{equation}
|A'-A'|\leq2^{33}\varepsilon^{-9}K^4|A'|.
\label{equation:graph-bsg-difference}
\end{equation}
\end{theorem}

\begin{proof}
Indicator uncertainty yields $r^2\geq\frac{1}{\alpha}$, hence $K\geq1$. The collision hierarchy and the edge-energy identity imply
$$
E_{\mathrm{edge}}(G,\pi)\geq\frac{N^2(2s)^2}{r^2}=\frac{(2s)^3}{\alpha r^2},
$$
and it remains to apply Theorem~\ref{theorem:bsg} to $A_\pi$ with $K=\alpha r^2$.
\end{proof}

\begin{remark}
The subset $A'$ is a set of oriented edges in ${\mathbb Z}_N^2$, and the conclusion concerns its difference set in that group. Passing from this conclusion to structure on vertex subsets or induced subgraphs requires further information.
\end{remark}

The conclusion is useful when the Fourier ratio is close to the uncertainty lower bound. Suppose, for example, that the graph has $s\asymp N$ edges and that a labeling satisfies $\operatorname{FR}(G,\pi)^2\asymp N$. Then $\alpha=\frac{2s}{N^2}\asymp\frac{1}{N}$, so that $K=\alpha\operatorname{FR}(G,\pi)^2\asymp1$, and Theorem \ref{theorem:graph-bsg} supplies a subset containing a positive proportion of the oriented edges and having a difference set of comparable size.

\subsection{Why large additive energy does not give a converse}
\label{subsection:failure-converse}

Can we reverse the preceding implication? A large fourth moment need not control the first moment, since many small Fourier coefficients can have a large sum, and density makes this easy to see. Indeed, if $A\subset H$ has density $\alpha=\frac{|A|}{|H|}$, then the trivial character alone accounts for
\begin{equation}
E_+(A)\geq\frac{|A|^4}{|H|}=\alpha|A|^3,
\label{equation:density-energy-floor}
\end{equation}
so every set of fixed positive density already has additive energy comparable to the maximal scale $|A|^3$. We still need a graph whose Fourier ratio is large, and the Paley graphs provide an explicit example.

\begin{theorem}[A dense family with bounded collision dimension]
\label{theorem:paley-counterexample}
Let $p\equiv1\pmod4$ be prime, and let $P_p$ be the graph on ${\mathbb Z}_p$ in which $x$ and $y$ are adjacent when $y-x$ is a nonzero square. Under the natural labeling, let $A_p$ be its directed edge set and $f_p=1_{A_p}$. Then
\begin{equation}
\operatorname{FR}_{\min}(P_p)=\operatorname{FR}(f_p)=\sqrt{\frac{p-1}{2p}}(1+\sqrt p),
\label{equation:paley-fr}
\end{equation}
\begin{equation}
D_4(f_p)=\frac{4p(p-1)}{p^2-2p+9}\longrightarrow4,
\label{equation:paley-collision}
\end{equation}
\begin{equation}
\frac{E_+(A_p)}{|A_p|^3}=\frac{1}{2}+\frac{4}{(p-1)^2}\longrightarrow\frac{1}{2},
\label{equation:paley-normalized-energy}
\end{equation}
and for every fixed $0<\eta<1$,
\begin{equation}
\frac{\kappa_\eta(f_p)}{p}\longrightarrow\max\{0,1-2\eta^2\}.
\label{equation:paley-approximation}
\end{equation}
In particular, bounded collision dimension does not imply bounded Fourier approximation complexity when the required squared relative error is less than $\frac{1}{2}$.
\end{theorem}

\begin{proof}
Write $\chi$ for the quadratic character of ${\mathbb Z}_p$, with $\chi(0)=0$. There are $\frac{p-1}{2}$ nonzero squares and the same number of nonsquares, the nonzero squares form an index-two subgroup of ${\mathbb Z}_p^\times$, so $\chi$ is multiplicative and $\sum_t\chi(t)=0$, and since $p\equiv1\pmod4$, $-1$ is a square, which makes the graph undirected. It has degree $d=\frac{p-1}{2}$.

Let us compute the Gauss sum needed for the spectrum. Put
$$
\tau=\sum_{t\in{\mathbb Z}_p}\chi(t)e^{-\frac{2\pi it}{p}}.
$$
The substitution $t\mapsto-t$ and $\chi(-1)=1$ show that $\tau$ is real, and substituting $a=ub$ in the expansion of $|\tau|^2$ leads to
$$
|\tau|^2=\sum_{u\neq0}\chi(u)\sum_{b\neq0}e^{-\frac{2\pi i(u-1)b}{p}}=(p-1)-\sum_{u\neq0,1}\chi(u)=p,
$$
so that $\tau=\pm\sqrt p$. If $S$ is the set of nonzero squares, then
$$
1_S(t)=\frac{1_{\{t\neq0\}}+\chi(t)}{2},
$$
and for $n\neq0$ the substitution $t\mapsto n^{-1}t$ therefore yields
$$
\sum_t1_S(t)e^{-\frac{2\pi int}{p}}=\frac{-1+\chi(n)\tau}{2}.
$$
The adjacency eigenvalues are consequently $\frac{p-1}{2}$ once, $\frac{-1+\sqrt p}{2}$ with multiplicity $\frac{p-1}{2}$, and $\frac{-1-\sqrt p}{2}$ with multiplicity $\frac{p-1}{2}$. These are the usual Paley eigenvalues, obtained here without a choice of sign for $\tau$.

Since $f_p(x,y)=1_S(y-x)$, a direct change of variables shows that
$$
\widehat f_p(m,n)=1_{\{m+n=0\}}\sum_t1_S(t)e^{-\frac{2\pi int}{p}},
$$
which is a weighted partial permutation matrix, so the labeling is sharp and minimizes the Fourier ratio. Summing the absolute eigenvalues we find that
$$
\mathcal E(P_p)=\frac{p-1}{2}(1+\sqrt p)
\text{ and }
|A_p|=\frac{p(p-1)}{2},
$$
which proves \eqref{equation:paley-fr}. The fourth spectral moment is
$$
\operatorname{tr}(M^4)=\frac{(p-1)^4+(p-1)(p^2+6p+1)}{16}=\frac{p(p-1)(p^2-2p+9)}{16},
$$
and Theorem \ref{theorem:extremizer-exact-energy} now proves \eqref{equation:paley-collision} and \eqref{equation:paley-normalized-energy}.

For the approximation claim, the principal Fourier coefficient has squared modulus $d^2$ and accounts for the fraction $\frac{p-1}{2p}$ of the total squared norm, while each of the remaining $p-1$ nonzero coefficients accounts for the fraction
$$
\frac{1+O(p^{-\frac{1}{2}})}{2p}
$$
of that norm, uniformly over those coefficients. For $1\leq k\leq p$, the fraction captured by the $k$ largest coefficients is therefore
$$
\frac{1}{2}+\frac{k-1}{2p}+O(p^{-\frac{1}{2}}),
$$
uniformly in $k$, and comparing this with the required fraction $1-\eta^2$ proves \eqref{equation:paley-approximation}. At $\eta^2=\frac{1}{2}$ the same estimate shows that the required number is $o(p)$, which is all that the asserted limit needs.
\end{proof}

\begin{remark}
Subtracting the mean removes the principal Fourier coefficient. Directly subtracting $d^4$ from the fourth spectral moment produces
$$
D_4\left(f_p-\frac{p-1}{2p}\right)=\frac{(p-1)(p+1)^2}{p^2+6p+1}\asymp p
\text{ and }
\operatorname{FR}\left(f_p-\frac{p-1}{2p}\right)^2=\frac{p(p-1)}{p+1}\asymp p,
$$
so both centered quantities are proportional to $p$, as is the number of frequencies needed to approximate the nonconstant part to a fixed relative error.
\end{remark}

We can push the normalized additive energy all the way to one while keeping the minimum Fourier ratio unbounded. The following construction does this for connected regular graphs, and taking a bipartite double cover then produces graphs of density tending to one half.

\begin{theorem}[Regular separation, including density one half]
\label{theorem:regular-separation}
There are connected regular graphs $H_p$ such that, uniformly over all cyclic labelings,
$$
\frac{E_{\mathrm{edge}}(H_p,\pi)}{(2|E(H_p)|)^3}\longrightarrow1,
D_4(H_p,\pi)\longrightarrow1
\text{ and }
\operatorname{FR}_{\min}(H_p)\longrightarrow\infty.
$$
There are also connected regular bipartite graphs $B_p$ of density tending to $\frac{1}{2}$, with cyclic labelings $\beta_p$ satisfying
$$
\frac{E_{\mathrm{edge}}(B_p,\beta_p)}{(2|E(B_p)|)^3}\longrightarrow1,
D_4(B_p,\beta_p)\longrightarrow2
\text{ and }
\operatorname{FR}_{\min}(B_p)\longrightarrow\infty.
$$
In the bipartite construction, the energy and collision limits hold for the labelings $\beta_p$.
\end{theorem}

\begin{proof}
Let $p\equiv1\pmod4$ tend to infinity through primes, choose odd integers $t_p\geq3$ tending to infinity, set $n_p=t_pp$, and let $H_p$ be the complement of the disjoint union of $t_p$ copies of $P_p$, so that $H_p$ is $d_p$-regular with
$$
d_p=n_p-\frac{p+1}{2}.
$$
All vertices in distinct blocks are adjacent, so $H_p$ is connected and contains a triangle, and its density $\frac{d_p}{n_p}$ tends to one. The density floor \eqref{equation:density-energy-floor} and the upper bound $E_+(A)\leq|A|^3$ prove the first energy limit uniformly, and the bounds $1\leq D_4\leq(\frac{n_p}{d_p})^2$ prove the collision limit.

On the subspace of vectors summing to zero in each Paley block, the adjacency eigenvalues of $H_p$ are $-1-\lambda$, where $\lambda$ runs over the nonprincipal Paley eigenvalues, while on the block-constant subspace the eigenvalues are $d_p$ once and $-\frac{p+1}{2}$ with multiplicity $t_p-1$. Consequently
$$
\mathcal E(H_p)=d_p+\frac{(t_p-1)(p+1)}{2}+\frac{t_p(p-1)\sqrt p}{2},
$$
and the graph-energy lower bound yields
$$
\operatorname{FR}_{\min}(H_p)\geq\frac{\mathcal E(H_p)}{\sqrt{n_pd_p}}\geq\frac{p-1}{2\sqrt p}\longrightarrow\infty.
$$

Now take the weak product $B_p=K_2\times H_p$, which has $2n_p$ vertices, degree $d_p$ and density $\frac{d_p}{2n_p}\to\frac{1}{2}$. To check connectedness, lift a path in $H_p$ to the bipartite double cover, and note that a detour around an odd cycle changes the parity of its length, so both lifts of every vertex can be reached. Since $n_p$ is odd, the product labeling by ${\mathbb Z}_2\times{\mathbb Z}_{n_p}$ is a cyclic labeling by the Chinese remainder theorem. The two oriented edges of $K_2$ form a coset in ${\mathbb Z}_2^2$, so their additive energy is $2^3$ and their collision dimension is $2$, and the tensorization proved in Proposition~\ref{proposition:weak-product-tensorization} shows that
$$
\frac{E_{\mathrm{edge}}(B_p,\beta_p)}{(2n_pd_p)^3}=\frac{E_{\mathrm{edge}}(H_p,\pi_p)}{(n_pd_p)^3}
\text{ and }
D_4(B_p,\beta_p)=2D_4(H_p,\pi_p).
$$
Finally, the adjacency spectrum of a weak product consists of the pairwise products of eigenvalues, hence
$$
\operatorname{FR}_{\min}(B_p)\geq\frac{2\mathcal E(H_p)}{\sqrt{2n_pd_p}}\geq\frac{p-1}{\sqrt{2p}}\longrightarrow\infty.
$$
\end{proof}

For $H_p$ the best constant approximation has vanishing relative error, while for $B_p$ the constant alone leaves squared relative error tending to $\frac{1}{2}$ and the two Fourier frequencies of its ambient complete bipartite graph give vanishing relative error. These examples show that the Fourier ratio can grow while approximation to a fixed error requires only a few frequencies, and for an example with growing approximation complexity we return to the Paley family and \eqref{equation:paley-approximation}.

\begin{proposition}[A density baseline under random labeling]
\label{proposition:random-labeling-baseline}
Let $G$ be a graph with $N$ vertices and density $\rho=\frac{2s}{N^2}$. If $\pi$ is chosen uniformly from all vertex labelings, then, with an absolute implied constant,
$$
{\mathbb E}_\pi E_{\mathrm{edge}}(G,\pi)=\rho^4N^6+O(N^5).
$$
Suppose that $G=G_N$ is a sequence with $\rho_N\geq c>0$. Then
$$
{\mathbb E}_\pi[D_4(G_N,\pi)^{-1}]=\rho_N^2+O_c(N^{-1})
\text{ and }
D_4(G_N,\pi)-\rho_N^{-2}\longrightarrow0\text{ in probability},
$$
and moreover
$$
\rho_N\leq\min_\pi\frac{E_{\mathrm{edge}}(G_N,\pi)}{(2s_N)^3}\leq\rho_N+O_c(N^{-1}).
$$
Thus no sequence with density bounded away from both zero and one can have normalized additive edge energy tending to one uniformly over all labelings.
\end{proposition}

\begin{proof}
There are $N^6$ ordered quadruples $((x_1,y_1),\ldots,(x_4,y_4))$ satisfying
$$
x_1+x_2=x_3+x_4
\text{ and }
y_1+y_2=y_3+y_4.
$$
For any prescribed pair among the eight scalar coordinates, imposing equality between that pair leaves exactly $N^5$ possibilities. Indeed, within one equation the equality leaves two free coordinates, while across the two equations one chooses the common coordinate and two more coordinates in each equation and then solves for the two remaining coordinates. Hence all but $O(N^5)$ additive quadruples have eight distinct scalar coordinates.

Let $T$ be the number of ordered quadruples of directed edges of $G$ whose eight endpoints are distinct. Among the $(2s)^4$ ordered quadruples of directed edges, those with any specified endpoint collision number at most $N^7$ once the edge constraints are dropped, and therefore $T=(2s)^4+O(N^7)$. For an additive quadruple with eight distinct scalar coordinates, a uniform random labeling assigns them to a uniformly random ordered eight-tuple of distinct vertices, so its probability of consisting of edges is
$$
\frac{T}{(N)_8}=\rho^4+O(N^{-1}),
$$
where $(N)_8=N(N-1)\cdots(N-7)$. Summing over the $N^6+O(N^5)$ such quadruples and bounding the exceptional contribution by $O(N^5)$ proves the expectation formula. The argument applies for sufficiently large $N$, which suffices for the estimates, and enlarging the absolute constant covers the smaller values of $N$.

The zero Fourier coefficient supplies the deterministic inequality $E_{\mathrm{edge}}(G,\pi)\geq\rho^4N^6$, which together with the edge-energy identity implies that
$$
D_4(G,\pi)^{-1}-\rho^2\geq0
\text{ and }
{\mathbb E}_\pi[D_4(G,\pi)^{-1}-\rho^2]=O(\rho^{-2}N^{-1}).
$$
Markov's inequality and $\rho\geq c$ give convergence of the inverse collision dimension, and then of the collision dimension itself, since both inverse quantities are at least $c^2$. Finally, dividing the deterministic lower bound and the expectation formula by $(2s)^3=\rho^3N^6$, and using the fact that a minimum is at most the average, proves the last pair of bounds.
\end{proof}

\begin{remark}
For graphs of density tending to $\frac{1}{2}$, the collision dimension tends to $4$ under a uniformly random labeling, and the naturally labeled Paley graphs have this same limit. Their additional feature is that the exact spectrum also forces a large number of Fourier frequencies when $\eta^2<\frac{1}{2}$.
\end{remark}

\section{Cyclic Cayley graphs}
\label{section:cayley}

For a cyclic Cayley graph, the adjacency function depends only on the difference of its two variables, so its two-dimensional Fourier transform reduces to the one-dimensional transform of the connection set. Let us work out what this says about additive energy and approximation.

Let $S\subset{\mathbb Z}_N$ be nonempty with $0\notin S$ and $S=-S$. The cyclic Cayley graph $G=\operatorname{Cay}({\mathbb Z}_N,S)$ has vertex set ${\mathbb Z}_N$, with $x\sim y$ if and only if $y-x\in S$. The condition $0\notin S$ excludes loops, and $S=-S$ makes the graph undirected. If $d=|S|$, then $G$ is $d$-regular and has $s=\frac{Nd}{2}$ undirected edges. Its directed edge set is
\begin{equation}
A_S=\{(x,x+t):x\in{\mathbb Z}_N,t\in S\},
\label{equation:cayley-edge-set}
\end{equation}
of cardinality $|A_S|=Nd$, and its adjacency function is
\begin{equation}
f_S(x,y)=1_S(y-x).
\label{equation:cayley-adjacency-function}
\end{equation}

\begin{theorem}[Exact additive edge energy]
\label{theorem:cayley-edge-energy}
Let $G=\operatorname{Cay}({\mathbb Z}_N,S)$ be as above. Then
\begin{equation}
E_+(A_S)=N^3E_+(S),
\label{equation:cayley-energy-identity}
\end{equation}
and consequently
\begin{equation}
D_4(f_S)=\frac{Nd^2}{E_+(S)}.
\label{equation:cayley-d4}
\end{equation}
\end{theorem}

\begin{proof}
The map $(x,t)\mapsto(x,x+t)$ is a group automorphism of ${\mathbb Z}_N^2$ that maps ${\mathbb Z}_N\times S$ onto $A_S$. Additive energy is invariant under group isomorphisms and multiplicative under Cartesian products, and $E_+({\mathbb Z}_N)=N^3$, so $E_+(A_S)=N^3E_+(S)$. Substitution into $D_4(f_S)=\frac{N^2(Nd)^2}{E_+(A_S)}$ completes the proof.
\end{proof}

The Fourier calculation from \cite{GuptaIosevich2026} also reduces the collision dimension to that of the connection set. For a function $h:{\mathbb Z}_N\to{\mathbb C}$ define
$$
\widehat h(n)=\frac{1}{\sqrt N}\sum_{t\in{\mathbb Z}_N}h(t)e^{\frac{-2\pi int}{N}}.
$$

\begin{proposition}[Fourier reduction to the connection set]
\label{proposition:cayley-fourier-reduction}
For the adjacency function \eqref{equation:cayley-adjacency-function},
\begin{equation}
\widehat f_S(m,n)=\sqrt N1_{\{m+n=0\}}\widehat{1_S}(n),
\label{equation:cayley-fourier-transform}
\end{equation}
and consequently
\begin{equation}
\operatorname{FR}(f_S)=\operatorname{FR}(1_S)
\label{equation:cayley-fr-reduction}
\end{equation}
and
\begin{equation}
D_4(f_S)=D_4(1_S).
\label{equation:cayley-d4-reduction}
\end{equation}
\end{proposition}

\begin{proof}
Substituting $y=x+t$ we obtain
$$
\widehat f_S(m,n)=\frac{1}{N}\left(\sum_xe^{-\frac{2\pi i(m+n)x}{N}}\right)\left(\sum_t1_S(t)e^{-\frac{2\pi int}{N}}\right)=\sqrt N1_{\{m+n=0\}}\widehat{1_S}(n),
$$
so the nonzero coefficient vectors agree up to the common factor $\sqrt N$, which cancels from both ratios.
\end{proof}

The same calculation produces the adjacency spectrum, and the resulting identity between the Fourier ratio and graph energy is due to \cite{GuptaIosevich2026}. For $\xi\in{\mathbb Z}_N$ define the normalized character
$$
\chi_\xi(x)=\frac{1}{\sqrt N}e^{\frac{2\pi i\xi x}{N}}.
$$
These normalized characters form an orthonormal basis of $\ell^2({\mathbb Z}_N)$, and if $A$ is the adjacency operator, then
$$
(A\chi_\xi)(x)=\sum_{t\in S}\chi_\xi(x+t)=\chi_\xi(x)\sum_{t\in S}e^{\frac{2\pi i\xi t}{N}}.
$$
Thus the adjacency eigenvalue corresponding to $\chi_\xi$ is
\begin{equation}
\theta_\xi=\sum_{t\in S}e^{\frac{2\pi i\xi t}{N}},
\label{equation:cayley-adjacency-eigenvalue}
\end{equation}
which is real because $S=-S$.

\begin{proposition}[Fourier ratio and graph energy, \cite{GuptaIosevich2026}]
\label{proposition:cayley-fr-graph-energy}
Let $G=\operatorname{Cay}({\mathbb Z}_N,S)$ with $d=|S|$. Then
\begin{equation}
\operatorname{FR}(f_S)=\frac{\mathcal E(G)}{\sqrt{Nd}}.
\label{equation:cayley-fr-graph-energy}
\end{equation}
\end{proposition}

\begin{proof}
The coefficient at $(-n,n)$ equals the real eigenvalue $\theta_n$, so $\|\widehat f_S\|_1=\mathcal E(G)$, while Plancherel shows that $\|\widehat f_S\|_2^2=Nd$.
\end{proof}

Combining Theorem \ref{theorem:cayley-edge-energy}, Proposition \ref{proposition:cayley-fr-graph-energy} and the collision hierarchy produces a relation between graph energy and additive energy.

\begin{corollary}[Graph energy versus additive energy]
\label{corollary:graph-energy-additive-energy}
Let $G=\operatorname{Cay}({\mathbb Z}_N,S)$ and let $d=|S|$. Then
\begin{equation}
\mathcal E(G)^2E_+(S)\geq N^2d^3,
\label{equation:graph-energy-additive-energy}
\end{equation}
or equivalently
$$
\mathcal E(G)\geq\frac{Nd^{\frac{3}{2}}}{\sqrt{E_+(S)}}.
$$
\end{corollary}

\begin{proof}
Substitute $D_4(f_S)=\frac{Nd^2}{E_+(S)}$ and $\operatorname{FR}(f_S)^2=\frac{\mathcal E(G)^2}{Nd}$ into $D_4(f_S)\leq\operatorname{FR}(f_S)^2$ and rearrange.
\end{proof}

Substituting these identities into the approximation bounds yields the following theorem.

\begin{theorem}[Cayley graph approximation complexity]
\label{theorem:cayley-approximation-complexity}
Let $G=\operatorname{Cay}({\mathbb Z}_N,S)$, let $d=|S|$, and let $0<\eta<1$. Then
\begin{equation}
\left\lceil(1-\eta^2)^2\frac{Nd^2}{E_+(S)}\right\rceil\leq\kappa_\eta(f_S)\leq\min\left\{N,\left\lfloor\frac{\frac{\mathcal E(G)^2}{Nd}-1}{\eta^2}\right\rfloor+1\right\}.
\label{equation:cayley-two-sided-complexity}
\end{equation}
\end{theorem}

\begin{proof}
Apply Corollary \ref{corollary:two-sided-general} and substitute $D_4(f_S)=\frac{Nd^2}{E_+(S)}$ and $\operatorname{FR}(f_S)^2=\frac{\mathcal E(G)^2}{Nd}$.
\end{proof}

\subsection{The cycle graph}

Let $N\geq5$ and consider the cycle graph $C_N=\operatorname{Cay}({\mathbb Z}_N,\{-1,1\})$, so that $d=2$. The sums of two elements of $S=\{-1,1\}$ are $-2$, $0$ and $2$, the sums $-2$ and $2$ each have one ordered representation, while $0$ has two, and since $N\geq5$ these three residues are distinct. Therefore
\begin{equation}
E_+(S)=1^2+2^2+1^2=6,
\label{equation:cycle-connection-energy}
\end{equation}
and Theorem \ref{theorem:cayley-edge-energy} yields
\begin{equation}
E_+(A_{C_N})=6N^3
\label{equation:cycle-edge-energy}
\end{equation}
and
\begin{equation}
D_4(f_{C_N})=\frac{N\cdot2^2}{6}=\frac{2N}{3}.
\label{equation:cycle-d4}
\end{equation}
The adjacency eigenvalues are $\lambda_j=2\cos(\frac{2\pi j}{N})$, so
$$
\mathcal E(C_N)=\sum_{j=0}^{N-1}2\left|\cos\left(\frac{2\pi j}{N}\right)\right|,
$$
and the Riemann-sum limit is
$$
\frac{\mathcal E(C_N)}{N}\longrightarrow\int_0^12|\cos(2\pi t)|dt=\frac{4}{\pi}.
$$
Consequently
\begin{equation}
\operatorname{FR}(f_{C_N})^2=\frac{\mathcal E(C_N)^2}{2N}\sim\frac{8}{\pi^2}N,
\label{equation:cycle-fr-asymptotic}
\end{equation}
and Theorem \ref{theorem:cayley-approximation-complexity} therefore implies that $\kappa_\eta(f_{C_N})\asymp_\eta N$ for every fixed $0<\eta<1$. More explicitly,
$$
\left\lceil\frac{2}{3}(1-\eta^2)^2N\right\rceil\leq\kappa_\eta(f_{C_N})\leq N,
$$
so for every fixed relative error less than one the number of required frequencies is proportional to $N$.

\subsection{Parity-labeled complete bipartite graphs}

Let $N$ be even and let $S$ be the set of odd residues modulo $N$, so that $G=\operatorname{Cay}({\mathbb Z}_N,S)$ is the parity-labeled graph $K_{\frac{N}{2},\frac{N}{2}}$. The set $S$ is a coset of the subgroup of even residues, and if a set $C$ is a coset of a subgroup with $|C|=d$, then every choice of three elements $c_1,c_2,c_3\in C$ determines $c_4=c_1+c_2-c_3\in C$, so that $E_+(C)=d^3$. Here $d=\frac{N}{2}$, hence $E_+(S)=(\frac{N}{2})^3$ and
$$
D_4(f_S)=\frac{N(\frac{N}{2})^2}{(\frac{N}{2})^3}=2.
$$
The adjacency spectrum of $K_{\frac{N}{2},\frac{N}{2}}$ consists of $\frac{N}{2}$, $-\frac{N}{2}$ and $0$ with multiplicity $N-2$, so $\mathcal E(G)=N$ and
$$
\operatorname{FR}(f_S)^2=\frac{N^2}{N(\frac{N}{2})}=2,
$$
which confirms that $D_4(f_S)=\operatorname{FR}(f_S)^2=2$. Since the two nonzero coefficients have equal modulus, the exact approximation complexity is one when $\eta^2\geq\frac{1}{2}$ and two when $\eta^2<\frac{1}{2}$.

\subsection{The complete graph as a Cayley graph}

The complete graph $K_N$ is $\operatorname{Cay}({\mathbb Z}_N,{\mathbb Z}_N\setminus\{0\})$, with $d=N-1$. From Example \ref{example:complete-graph-centered},
$$
D_4(f)=\frac{N(N-1)}{N^2-3N+3},
$$
and using \eqref{equation:cayley-d4} we obtain the exact connection-set energy
$$
E_+({\mathbb Z}_N\setminus\{0\})=(N-1)(N^2-3N+3),
$$
which also follows from a direct representation count. As in the centered calculation, the zero-frequency coefficient explains why $D_4(f)$ tends to one.

\section{Weak products and tensorization}
\label{section:weak-products}

The weak graph product separates the variables, and as a result the additive edge energy, the collision dimension and the Fourier ratio all multiply. We use the same definitions when the vertices are labeled by an arbitrary finite abelian group, with the unitary Fourier transform on that group, and for the product graph the labeling group will be ${\mathbb Z}_{N_1}\times{\mathbb Z}_{N_2}$.

Let $G$ and $H$ be graphs with at least one edge each and with labelings $\pi$ by ${\mathbb Z}_{N_1}$ and $\sigma$ by ${\mathbb Z}_{N_2}$. Their weak product $G\times H$ has vertex set $V(G)\times V(H)$, with $(g_1,h_1)\sim(g_2,h_2)$ if and only if $g_1\sim g_2$ in $G$ and $h_1\sim h_2$ in $H$, and we use the product labeling by ${\mathbb Z}_{N_1}\times{\mathbb Z}_{N_2}$. If $N_1$ and $N_2$ are coprime, the Chinese remainder theorem transports this to a cyclic labeling by ${\mathbb Z}_{N_1N_2}$.

\begin{proposition}[Tensorization under weak products]
\label{proposition:weak-product-tensorization}
Under the product labeling,
\begin{equation}
E_{\mathrm{edge}}(G\times H,\pi\times\sigma)=E_{\mathrm{edge}}(G,\pi)E_{\mathrm{edge}}(H,\sigma),
\label{equation:weak-product-energy-tensor}
\end{equation}
\begin{equation}
D_4(G\times H,\pi\times\sigma)=D_4(G,\pi)D_4(H,\sigma),
\label{equation:weak-product-d4-tensor}
\end{equation}
and
\begin{equation}
\operatorname{FR}(G\times H,\pi\times\sigma)=\operatorname{FR}(G,\pi)\operatorname{FR}(H,\sigma).
\label{equation:weak-product-fr-tensor}
\end{equation}
Consequently, for every $0<\eta<1$,
$$
\kappa_\eta(G\times H,\pi\times\sigma)\geq\lceil(1-\eta^2)^2D_4(G,\pi)D_4(H,\sigma)\rceil.
$$
\end{proposition}

\begin{proof}
After the coordinate permutation $((x_1,x_2),(y_1,y_2))\mapsto((x_1,y_1),(x_2,y_2))$, the directed edge set of the weak product is the Cartesian product $A_{\pi\times\sigma}=A_\pi\times A_\sigma$, and an additive quadruple in this set is exactly a pair consisting of an additive quadruple in $A_\pi$ and an additive quadruple in $A_\sigma$, which proves \eqref{equation:weak-product-energy-tensor}. The adjacency function is the tensor product $f_{\pi\times\sigma}=f_\pi\otimes f_\sigma$, meaning that
$$
(f_\pi\otimes f_\sigma)((x_1,x_2),(y_1,y_2))=f_\pi(x_1,y_1)f_\sigma(x_2,y_2),
$$
and separating the sums in the Fourier transform shows that $\widehat f_{\pi\times\sigma}=\widehat f_\pi\otimes\widehat f_\sigma$. The $\ell^1$, $\ell^2$ and $\ell^4$ norms multiply under tensor products, which proves \eqref{equation:weak-product-d4-tensor} and \eqref{equation:weak-product-fr-tensor}, and the approximation lower bound follows from Theorem \ref{theorem:collision-approximation-general}.
\end{proof}

\begin{remark}
Weak products of sharp graphs of coprime orders remain sharp \cite{GuptaIosevichEtAl2026}, and the preceding proposition also determines their additive edge energy and collision dimension. It supplies a lower bound for approximation complexity, although it does not assert that approximation complexity multiplies.
\end{remark}

\section{Stability and its necessary scale}
\label{section:adjacency-stability}

How much can we conclude when a labeling comes close to the graph-energy bound? We first approximate its Fourier transform by a weighted partial permutation matrix, with a bound that involves both the deficit and a spectral scale, and the example that follows shows why the latter cannot be omitted, even for regular graphs admitting a sharp labeling.

\begin{theorem}[A quantitative nuclear-norm estimate]
\label{theorem:nuclear-stability}
Let $B$ be a complex $n\times n$ matrix, and write
$$
L(B)=\sum_{i,j}|B_{ij}|
\text{ and }
b_\infty=\max_{i,j}|B_{ij}|.
$$
There is a weighted partial permutation matrix $W$ such that
\begin{equation}
\|B-W\|_{\mathrm F}^2\leq2b_\infty(L(B)-\|B\|_*),
\label{equation:nuclear-stability}
\end{equation}
and $W$ can be obtained by retaining the entries of $B$ along a single permutation and setting all other entries to zero.
\end{theorem}

\begin{proof}
Choose a singular value decomposition $B=V\Sigma Z^*$ and put $Q=VZ^*$. Then $Q$ is unitary and
$$
\|B\|_*=\operatorname{Re}\sum_{i,j}\overline{Q_{ij}}B_{ij}\leq\sum_{i,j}|Q_{ij}||B_{ij}|,
$$
so if $\Delta=L(B)-\|B\|_*$, then
\begin{equation}
\sum_{i,j}|B_{ij}|(1-|Q_{ij}|)\leq\Delta.
\label{equation:polar-deficit}
\end{equation}
The matrix $P=(|Q_{ij}|^2)$ is doubly stochastic, meaning that its entries are nonnegative and every row and column sums to one, and we recall why it is a convex combination of permutation matrices. For a set $I$ of rows, let $\mathcal N(I)$ be the set of columns containing a positive entry in one of those rows. Then
$$
|I|=\sum_{i\in I,j}P_{ij}\leq|\mathcal N(I)|,
$$
so Hall's marriage theorem supplies a perfect matching among the positive entries, meaning a choice of one positive entry in every row and every column. Subtract the smallest entry along such a matching times its permutation matrix, and unless this ends the procedure, normalize the remainder to be doubly stochastic and repeat. At least one positive entry disappears at each step, so the process terminates and produces the claimed convex combination.

Choose a random permutation $\sigma$ from this convex combination and let $W_\sigma$ retain the entries $B_{i,\sigma(i)}$. Then
$$
{\mathbb E}\|B-W_\sigma\|_{\mathrm F}^2=\sum_{i,j}|B_{ij}|^2(1-|Q_{ij}|^2),
$$
and since $1-t^2\leq2(1-t)$ for $0\leq t\leq1$, we obtain
$$
{\mathbb E}\|B-W_\sigma\|_{\mathrm F}^2\leq2b_\infty\sum_{i,j}|B_{ij}|(1-|Q_{ij}|)\leq2b_\infty\Delta.
$$
At least one permutation realizes an error no larger than this average.
\end{proof}

\begin{corollary}[Adjacency stability with a spectral scale factor]
\label{corollary:adjacency-stability-scale}
Let $\delta\geq0$, and suppose that a labeling $\pi$ of a simple graph $G$ with $N$ vertices and $s>0$ edges satisfies
$$
\operatorname{FR}(G,\pi)\leq(1+\delta)\frac{\mathcal E(G)}{\sqrt{2s}}.
$$
There is a weighted partial permutation matrix $W$ such that
\begin{equation}
\frac{\|\widehat M_\pi-W\|_{\mathrm F}^2}{\|\widehat M_\pi\|_{\mathrm F}^2}\leq\min\left\{1,\frac{2\delta\mathcal E(G)}{N}\right\}\leq\min\left\{1,2\delta\sqrt{\frac{2s}{N}}\right\}.
\label{equation:adjacency-stability-scale}
\end{equation}
In particular, the estimate is uniform over graph families of bounded average degree.
\end{corollary}

\begin{proof}
Every Fourier coefficient of the adjacency indicator has modulus at most $\frac{2s}{N}$, with equality at $(0,0)$, so $b_\infty=\frac{2s}{N}$, and the hypothesis says that $L(\widehat M_\pi)-\|\widehat M_\pi\|_*\leq\delta\mathcal E(G)$. Apply Theorem~\ref{theorem:nuclear-stability} and divide by $2s$. The zero matrix is also an admissible approximation, which accounts for the minimum with one, and finally the Cauchy--Schwarz inequality yields $\mathcal E(G)\leq\sqrt{N\sum_j\lambda_j^2}=\sqrt{2Ns}$.
\end{proof}

\begin{proposition}[No stability modulus depending only on relative deficit]
\label{proposition:no-uniform-adjacency-stability}
There are regular graphs $G_p$ and labelings $\pi_p$ such that
$$
\frac{\operatorname{FR}(G_p,\pi_p)}{\mathcal E(G_p)/\sqrt{2|E(G_p)|}}\longrightarrow1,
$$
but
\begin{equation}
\inf_W\frac{\|\widehat M_{\pi_p}-W\|_{\mathrm F}^2}{\|\widehat M_{\pi_p}\|_{\mathrm F}^2}\longrightarrow\frac{1}{20},
\label{equation:no-uniform-stability}
\end{equation}
where the infimum is over all weighted partial permutation matrices. Each abstract graph $G_p$ also admits a labeling attaining the graph-energy bound exactly. In particular, no estimate $C\delta^c$ with absolute $C,c>0$ follows from the relative deficit $\delta$ alone.
\end{proposition}

\begin{proof}
Let $p\equiv1\pmod4$ tend to infinity through primes, and let $H_p$ be the Paley graph on ${\mathbb Z}_p$, of degree $d=\frac{p-1}{2}$. Its nonprincipal adjacency eigenvalues are $\frac{-1+\sqrt p}{2}$ and $\frac{-1-\sqrt p}{2}$, each of multiplicity $\frac{p-1}{2}$, hence
\begin{equation}
\sum_{j\ne0}|\lambda_j(H_p)|=\frac{(p-1)\sqrt p}{2},
\label{equation:paley-nonprincipal-energy-stability}
\end{equation}
by the same Paley spectral formulas used earlier in the paper.

Let $T=K_{2,2}$ with its vertices labeled $0,1,2,3$ and bipartition $\{0,1\}\sqcup\{2,3\}$, and define $G_p=T[H_p]$, the lexicographic product in which every vertex of $T$ is replaced by a copy of $H_p$ and two copies are joined completely precisely when their vertices in $T$ are adjacent, as in Figure~\ref{figure:lexicographic}. Its adjacency matrix is
$$
M_p=A_T\otimes J_p+I_4\otimes A_{H_p}.
$$
It has $4p$ vertices and degree $2p+d$, so
\begin{equation}
\|M_p\|_{\mathrm F}^2=4p(2p+d)=10p^2-2p.
\label{equation:lexicographic-frobenius}
\end{equation}
We use the cyclic labeling supplied by ${\mathbb Z}_{4p}\cong{\mathbb Z}_4\times{\mathbb Z}_p$, under which the cyclic Fourier transform agrees with the product transform up to permutations of the frequency indices.

\begin{figure}[htb]
\centering
\begin{tikzpicture}[x=1cm,y=1cm]
\node[draw,circle,minimum size=1.3cm,font=\small] (v0) at (0,2) {$H_p$};
\node[draw,circle,minimum size=1.3cm,font=\small] (v1) at (0,0) {$H_p$};
\node[draw,circle,minimum size=1.3cm,font=\small] (v2) at (3.5,2) {$H_p$};
\node[draw,circle,minimum size=1.3cm,font=\small] (v3) at (3.5,0) {$H_p$};
\node[font=\footnotesize] at (-1.1,2) {$0$};
\node[font=\footnotesize] at (-1.1,0) {$1$};
\node[font=\footnotesize] at (4.6,2) {$2$};
\node[font=\footnotesize] at (4.6,0) {$3$};
\draw[very thick,double] (v0) -- (v2);
\draw[very thick,double] (v0) -- (v3);
\draw[very thick,double] (v1) -- (v2);
\draw[very thick,double] (v1) -- (v3);
\node[font=\footnotesize,align=center] at (1.75,-1.2) {a double line joins two copies completely};
\begin{scope}[xshift=8cm,yshift=0cm]
\draw (0,0) rectangle (3.6,3.6);
\draw (0,0.9) -- (3.6,0.9);
\draw (0,1.8) -- (3.6,1.8);
\draw (0,2.7) -- (3.6,2.7);
\draw (0.9,0) -- (0.9,3.6);
\draw (1.8,0) -- (1.8,3.6);
\draw (2.7,0) -- (2.7,3.6);
\fill[gray!30] (0,2.7) rectangle (0.9,3.6);
\fill[gray!30] (0.9,1.8) rectangle (1.8,2.7);
\fill[gray!30] (1.8,0.9) rectangle (2.7,1.8);
\fill[gray!30] (2.7,0) rectangle (3.6,0.9);
\node[font=\footnotesize] at (0.45,3.15) {$A_{H_p}$};
\node[font=\footnotesize] at (1.35,2.25) {$A_{H_p}$};
\node[font=\footnotesize] at (2.25,1.35) {$A_{H_p}$};
\node[font=\footnotesize] at (3.15,0.45) {$A_{H_p}$};
\node[font=\footnotesize] at (2.25,3.15) {$J_p$};
\node[font=\footnotesize] at (3.15,3.15) {$J_p$};
\node[font=\footnotesize] at (2.25,2.25) {$J_p$};
\node[font=\footnotesize] at (3.15,2.25) {$J_p$};
\node[font=\footnotesize] at (0.45,1.35) {$J_p$};
\node[font=\footnotesize] at (1.35,1.35) {$J_p$};
\node[font=\footnotesize] at (0.45,0.45) {$J_p$};
\node[font=\footnotesize] at (1.35,0.45) {$J_p$};
\node[font=\footnotesize] at (1.35,3.15) {$0$};
\node[font=\footnotesize] at (0.45,2.25) {$0$};
\node[font=\footnotesize] at (3.15,1.35) {$0$};
\node[font=\footnotesize] at (2.25,0.45) {$0$};
\node[font=\footnotesize] at (1.8,-0.4) {$M_p=A_T\otimes J_p+I_4\otimes A_{H_p}$};
\end{scope}
\end{tikzpicture}
\caption{The lexicographic product $G_p=K_{2,2}[H_p]$ of Proposition~\ref{proposition:no-uniform-adjacency-stability}, with the four Paley blocks on the left and the block structure of its adjacency matrix on the right. In the Fourier conjugate $C_p$, every nonzero inner frequency contributes a scalar block $\lambda_j(H_p)I_4$, and only the zero inner frequency contributes the $4\times4$ block $K_p$ whose $2\times2$ part \eqref{equation:bad-two-by-two-block} cannot be approximated well by a partial permutation.}
\label{figure:lexicographic}
\end{figure}

For the calculation, use the Fourier conjugation $C_p=(U_4\otimes U_p)M_p(U_4\otimes U_p)^*$. The two-dimensional transform differs from $C_p$ by a column permutation, which preserves all the norms in question and the distance to weighted partial permutation matrices. Since $H_p$ is circulant and $U_pJ_pU_p^*=pE_{00}$, every nonzero inner frequency contributes the diagonal block $\lambda_j(H_p)I_4$, while the zero inner frequency contributes
$$
K_p=pU_4A_TU_4^*+dI_4.
$$
The two scalar blocks of $U_4A_TU_4^*$, at the frequencies $0$ and $2$, are $2$ and $0$, and on the frequencies $1,3$ its block is
$$
\begin{pmatrix}-1&i\\-i&-1\end{pmatrix}.
$$
Thus $K_p$ has the scalar blocks $2p+d$ and $d$ together with the block
\begin{equation}
\begin{pmatrix}d-p&ip\\-ip&d-p\end{pmatrix},
\label{equation:bad-two-by-two-block}
\end{equation}
and its eigenvalues are $2p+d$, $d$, $d$ and $d-2p$, so that
$$
\|K_p\|_*=4p+2d=5p-1
\text{ and }
L(K_p)=6p.
$$
It follows from \eqref{equation:paley-nonprincipal-energy-stability} that
$$
\mathcal E(G_p)=2(p-1)\sqrt p+5p-1
\text{ and }
L(C_p)=\mathcal E(G_p)+p+1,
$$
so the relative deficit is exactly
$$
\delta_p=\frac{p+1}{2(p-1)\sqrt p+5p-1}\longrightarrow0.
$$

For any prescribed support, the best Frobenius approximation keeps the original entries on that support. In the block \eqref{equation:bad-two-by-two-block} the off-diagonal entries have modulus $p$, which exceeds the diagonal modulus $p-d$, so the optimal partial permutation keeps the two off-diagonal entries and loses the squared norm $2(p-d)^2=\frac{(p+1)^2}{2}$. All other blocks are diagonal and can be retained completely, and allowing entries between distinct blocks cannot improve the approximation, because all such entries of $C_p$ are zero, so that removing them preserves feasibility and never increases the error. Hence
$$
\inf_W\|C_p-W\|_{\mathrm F}^2=\frac{(p+1)^2}{2},
$$
and together with \eqref{equation:lexicographic-frobenius} this proves that
$$
\inf_W\frac{\|C_p-W\|_{\mathrm F}^2}{\|C_p\|_{\mathrm F}^2}=\frac{(p+1)^2}{4p(5p-1)}\longrightarrow\frac{1}{20}.
$$
Finally, relabeling $T$ so that its bipartition consists of the even and odd residues makes it circulant on ${\mathbb Z}_4$. The same product construction is then a Cayley graph on ${\mathbb Z}_4\times{\mathbb Z}_p\cong{\mathbb Z}_{4p}$, and this labeling is sharp, which establishes the last assertion without changing the underlying graph.
\end{proof}

\section{Spectral projectors and harmonic collision dimension}
\label{section:projectors}

The same parameters apply to orthogonal projector kernels. For a graph Laplacian $L=D-A$, where $D$ is the diagonal matrix of vertex degrees and $A$ is the adjacency matrix, write $\Pi_\lambda$ for the orthogonal projector onto the eigenspace of $\lambda$, of rank $m(\lambda)$. If $u_1,\ldots,u_m$ is an orthonormal basis of that eigenspace, then
$$
\Pi_\lambda(x,y)=\sum_{j=1}^mu_j(x)\overline{u_j(y)},
$$
and we refer to \cite{Chung1997} for the graph Laplacian and its spectral decomposition. We regard this matrix as a function on ${\mathbb Z}_N^2$, so that $D_4(\Pi_\lambda)$ and $\operatorname{FR}(\Pi_\lambda)$ are defined exactly as for arbitrary functions above, and minimizing the latter over vertex labelings defines $\operatorname{FR}_{\min}(\Pi_\lambda)$.

The lower bound $\operatorname{FR}_{\min}(\Pi_\lambda)\geq\sqrt{m(\lambda)}$ and its attainment for cyclic Cayley graphs are from \cite{GuptaIosevich2026}. We refine it by collision dimension and prove equality and stability statements, and in fact the arguments apply to any nonzero orthogonal projector, whether or not it arises from a graph. As before, the two-dimensional transform is $\widehat\Pi=U\Pi U$, whereas $U\Pi U^*$ is Fourier conjugation, and the two differ by the column permutation $R_{a,b}=1_{\{a+b=0\}}$.

\begin{theorem}[Collision refinement of the projector multiplicity bound]
\label{theorem:projector-hierarchy}
Let $\Pi_\lambda$ be a Laplacian spectral projector of rank $m=m(\lambda)\geq1$. Then
\begin{equation}
m\leq D_4(\Pi_\lambda)\leq\operatorname{FR}(\Pi_\lambda)^2.
\label{equation:projector-hierarchy}
\end{equation}
Consequently, if $0<\eta<1$ and a trigonometric polynomial $P$ on ${\mathbb Z}_N^2$ uses at most $k$ frequencies and satisfies $\|\Pi_\lambda-P\|_2<\eta\|\Pi_\lambda\|_2$, then
\begin{equation}
k>\max\{(1-\eta^2)^2D_4(\Pi_\lambda),(1-\eta^2)m(\lambda)\}.
\label{equation:projector-approximation-lower}
\end{equation}
\end{theorem}

\begin{proof}
The transformed projector has the same singular values as $\Pi_\lambda$, so $\|\widehat\Pi_\lambda\|_2^2=m$ and $\|\widehat\Pi_\lambda\|_{\mathrm{op}}=1$. In particular every entry has modulus at most one, so that $\|\widehat\Pi_\lambda\|_4^4\leq\|\widehat\Pi_\lambda\|_2^2=m$ and hence $D_4(\Pi_\lambda)\geq m$. The collision hierarchy supplies the upper inequality and its approximation obstruction, and since retaining $k$ coefficients captures squared norm at most $k$, an error strictly below $\eta\sqrt m$ also requires $k>(1-\eta^2)m$.
\end{proof}

\begin{proposition}[Equality and cyclic translation invariance]
\label{proposition:projector-equality-characterization}
Let $\Pi$ be a nonzero orthogonal projector of rank $m$, regarded as a kernel on ${\mathbb Z}_N^2$. For this fixed labeling, the following conditions are equivalent.
\begin{enumerate}
\item $D_4(\Pi)=m$.
\item $\operatorname{FR}(\Pi)^2=m$.
\item $\widehat\Pi$ has exactly $m$ nonzero entries, all of modulus one, in distinct rows and columns.
\item $U\Pi U^*$ is diagonal, with $m$ ones and $N-m$ zeros.
\item The range of $\Pi$ is spanned by a subset of the cyclic Fourier characters.
\item $\Pi$ is circulant, or equivalently its range is invariant under cyclic translation.
\end{enumerate}
Whenever these conditions hold, for every $0<\eta<1$ we have
$$
m=D_4(\Pi)=\operatorname{FR}(\Pi)^2=|\operatorname{supp}(\widehat\Pi)|
\text{ and }
\kappa_\eta(\Pi)=\lceil(1-\eta^2)m\rceil.
$$
\end{proposition}

\begin{proof}
Every entry of $\widehat\Pi$ has modulus at most one, and its squared Frobenius norm is $m$. Thus $D_4(\Pi)=m$ holds exactly when each entry has modulus zero or one, in which case there must be $m$ entries of modulus one, and since every row and column has Euclidean norm at most the operator norm, which is one, these entries occupy distinct rows and columns. This proves the equivalence of the first and third conditions. The third implies $\|\widehat\Pi\|_1=m$ and hence the second, and conversely the second implies the first by the collision hierarchy.

It remains to use the additional fact that the kernel is a projector. Let $C=U\Pi U^*$ and $R_{a,b}=1_{\{a+b=0\}}$. The Fourier matrix satisfies $U^2=R$, so $\widehat\Pi=CR$, and $C$ differs from $\widehat\Pi$ only by a column permutation. Moreover, $C$ is an orthogonal projector, so its diagonal entries satisfy
$$
C_{ii}=(C^2)_{ii}=\sum_j|C_{ij}|^2.
$$
If $C$ has at most one nonzero entry in each row, a nonzero off-diagonal entry would force the diagonal entry in that row to vanish, contradicting this identity, so $C$ is diagonal, and its diagonal entries are zero or one because $C^2=C$, which proves the fourth condition. Conversely, a diagonal $C$ with $m$ ones makes $CR$ satisfy the third condition.

The fourth and fifth conditions are equivalent by the definition of the Fourier matrix. If the character set is $\Gamma$, then
$$
\Pi(x,y)=\frac{1}{N}\sum_{\xi\in\Gamma}e^{\frac{2\pi i\xi(x-y)}{N}},
$$
which depends only on $x-y$ and is therefore circulant, and conversely circulant matrices are diagonalized by the Fourier basis. Finally, cyclic translation has the distinct Fourier characters as its eigenvectors, with distinct eigenvalues, and a translation-invariant subspace is spanned by a subset of these eigenvectors, because the projections onto individual characters are polynomials in the translation operator and preserve the subspace. This proves the last equivalence.

Under these conditions there are exactly $m$ nonzero Fourier coefficients of equal modulus, and retaining $k$ of them leaves the squared relative error $1-\frac{k}{m}$ for $0\leq k\leq m$, which proves the formula for $\kappa_\eta$.
\end{proof}

For a rank-one projector $\Pi=uu^*$ with $\|u\|_2=1$ and $v=Uu$, we have
$$
D_4(\Pi)=\left(\sum_k|v_k|^4\right)^{-2},
$$
so the equality $D_4(\Pi)=1$ means concentration on one Fourier mode. More generally, $D_4(\Pi)=m$ means that the Fourier transform of the projector has the smallest possible support, and the Fourier characters spanning its range each have constant modulus in vertex coordinates.

The proof of the equality criterion suggests a quantitative statement. We keep the $m$ Fourier coordinates on which the diagonal of $U\Pi U^*$ is largest, and this produces a nearby circulant projector, which need not belong to the spectral decomposition of the original graph.

\begin{theorem}[Stability of the projector collision bound]
\label{theorem:projector-stability}
Let $\Pi$ be a rank-$m$ orthogonal projector on ${\mathbb C}^N$ with $m\geq1$. There exists a circulant rank-$m$ orthogonal projector $Q$ such that
\begin{equation}
\frac{\|\Pi-Q\|_{\mathrm F}^2}{\|\Pi\|_{\mathrm F}^2}\leq2\left(1-\frac{m}{D_4(\Pi)}\right).
\label{equation:projector-stability}
\end{equation}
In particular, if $\delta\geq0$ and $D_4(\Pi)\leq(1+\delta)m$, then
$$
\|\Pi-Q\|_{\mathrm F}\leq\sqrt{\frac{2\delta}{1+\delta}}\|\Pi\|_{\mathrm F}.
$$
\end{theorem}

\begin{proof}
Again put $C=U\Pi U^*$ and write $q_i=C_{ii}$. Since $C$ is a rank-$m$ orthogonal projector,
$$
0\leq q_i\leq1,\sum_iq_i=m
\text{ and }
\sum_j|C_{ij}|^2=q_i,
$$
and consequently
\begin{equation}
\sum_{i,j}|C_{ij}|^4\leq\sum_iq_i^2.
\label{equation:projector-diagonal-fourth}
\end{equation}
Choose a set $I$ of $m$ indices with the largest $q_i$. Then
\begin{equation}
\sum_iq_i^2\leq\sum_{i\in I}q_i.
\label{equation:projector-rounding}
\end{equation}
To verify this, note that the case $m=N$ is immediate, and otherwise choose a number $a$ between the largest $q_i$ outside $I$ and the smallest $q_i$ in $I$. The identity $\sum_{i\notin I}q_i=\sum_{i\in I}(1-q_i)$ yields
$$
\sum_{i\notin I}q_i^2\leq a\sum_{i\notin I}q_i=a\sum_{i\in I}(1-q_i)\leq\sum_{i\in I}q_i(1-q_i),
$$
which is \eqref{equation:projector-rounding}.

Let $D$ be the diagonal projector with ones at the indices in $I$, and set $Q=U^*DU$, which is a circulant rank-$m$ projector. Unitary invariance of the Frobenius norm yields
$$
\|\Pi-Q\|_{\mathrm F}^2=\|C-D\|_{\mathrm F}^2=2m-2\sum_{i\in I}q_i,
$$
and the two preceding inequalities therefore imply
$$
\|\Pi-Q\|_{\mathrm F}^2\leq2m-2\sum_{i,j}|C_{ij}|^4=2m-\frac{2m^2}{D_4(\Pi)},
$$
where the last identity uses the fact that $C$ and $\widehat\Pi$ differ by a column permutation. Dividing by $\|\Pi\|_{\mathrm F}^2=m$ proves \eqref{equation:projector-stability}, and the asserted bound under $D_4(\Pi)\leq(1+\delta)m$ follows at once.
\end{proof}

\begin{remark}
The lower bound for the projector Fourier ratio is from \cite{GuptaIosevich2026}. Collision dimension provides the intermediate quantity in \eqref{equation:projector-hierarchy}, and its deficit controls the distance to a circulant projector. Translation invariance of one eigenspace need not imply translation invariance of the other eigenspaces, so the equality criterion for one projector does not characterize cyclic Cayley graphs.
\end{remark}

We can also recover the centered adjacency matrix from the nonconstant Laplacian eigenspaces.

\begin{proposition}[Centered adjacency and Laplacian projectors]
\label{proposition:centered-harmonic-decomposition}
Let $G$ be a connected $d$-regular graph with $N\geq2$ vertices. For each distinct Laplacian eigenvalue $\lambda$, let $\Pi_\lambda$ be the corresponding spectral projector and let $m(\lambda)$ be its multiplicity. Then
\begin{equation}
A-\frac{d}{N}J=\sum_{\lambda\neq0}(d-\lambda)\Pi_\lambda,
\label{equation:centered-harmonic-decomposition}
\end{equation}
and consequently
\begin{equation}
\operatorname{FR}^{\circ}(G,\pi)\leq\frac{\sum_{\lambda\neq0}|d-\lambda|\sqrt{m(\lambda)}\operatorname{FR}(\Pi_\lambda)}{\left(\sum_{\lambda\neq0}(d-\lambda)^2m(\lambda)\right)^{\frac{1}{2}}}.
\label{equation:centered-fr-projector-bound}
\end{equation}
\end{proposition}

\begin{proof}
Since $L=dI-A$ and $L=\sum_\lambda\lambda\Pi_\lambda$, the spectral decomposition of $A=dI-L$ is $A=\sum_\lambda(d-\lambda)\Pi_\lambda$. For every graph the constant vector lies in the zero eigenspace of $L$, and connectedness implies that this zero eigenspace is one-dimensional, so its orthogonal projector is $\Pi_0=\frac{J}{N}$. The coefficient of this projector in the preceding sum is $d$, which proves \eqref{equation:centered-harmonic-decomposition}. Taking Fourier transforms and using the triangle inequality,
$$
\left\|\widehat{A-\frac{d}{N}J}\right\|_1\leq\sum_{\lambda\neq0}|d-\lambda|\|\widehat\Pi_\lambda\|_1,
$$
and since $\|\widehat\Pi_\lambda\|_2=\sqrt{m(\lambda)}$ we have $\|\widehat\Pi_\lambda\|_1=\sqrt{m(\lambda)}\operatorname{FR}(\Pi_\lambda)$. For distinct eigenvalues $\lambda$ and $\mu$ the projectors have orthogonal ranges and satisfy $\Pi_\lambda\Pi_\mu=0$, so they are orthogonal in the Frobenius inner product, and since $\|\Pi_\lambda\|_2^2=m(\lambda)$,
$$
\left\|A-\frac{d}{N}J\right\|_2^2=\sum_{\lambda\neq0}(d-\lambda)^2m(\lambda).
$$
Dividing the preceding $\ell^1$ estimate by the square root of this identity proves \eqref{equation:centered-fr-projector-bound}.
\end{proof}

For a cyclic Cayley graph, every Laplacian eigenspace is spanned by cyclic Fourier characters, so the projector inequalities are equalities, and the corresponding Fourier-ratio identity was proved in \cite{GuptaIosevich2026}.

\begin{theorem}[Collision dimension of cyclic Cayley projectors]
\label{theorem:cayley-projector-exact}
Let $G=\operatorname{Cay}({\mathbb Z}_N,S)$, put $d=|S|$, and let $\lambda$ be a Laplacian eigenvalue. The adjacency eigenvalue at frequency $\xi$ is $\theta_\xi$ from \eqref{equation:cayley-adjacency-eigenvalue}, and the corresponding Laplacian eigenvalue is
$$
\mu_\xi=d-\theta_\xi=d-\sum_{t\in S}e^{\frac{2\pi i\xi t}{N}}.
$$
Define $S_\lambda=\{\xi\in{\mathbb Z}_N:\mu_\xi=\lambda\}$. Then $m(\lambda)=|S_\lambda|$, and under the natural group labeling
\begin{equation}
D_4(\Pi_\lambda)=\operatorname{FR}(\Pi_\lambda)^2=m(\lambda).
\label{equation:cayley-projector-equality}
\end{equation}
\end{theorem}

\begin{proof}
The eigenspace is spanned by the characters indexed by $S_\lambda$, so
$$
\Pi_\lambda(x,y)=\frac{1}{N}\sum_{\xi\in S_\lambda}e^{\frac{2\pi i\xi(x-y)}{N}}
\text{ and }
\widehat\Pi_\lambda(m,n)=1_{\{m\in S_\lambda\}}1_{\{n=-m\}}.
$$
There are exactly $m(\lambda)$ coefficients, each equal to one, so both ratios equal $m(\lambda)$ after the Fourier ratio is squared, as asserted.
\end{proof}

\begin{example}
For the cycle $C_N$ with $N\geq5$, the Laplacian eigenvalues are $2-2\cos(\frac{2\pi k}{N})$. Except for the frequency $k=0$ and, when $N$ is even, the frequency $k=\frac{N}{2}$, the eigenspaces have multiplicity two, corresponding to the pair of frequencies $\{k,-k\}$, so $D_4(\Pi_\lambda)=\operatorname{FR}(\Pi_\lambda)^2=2$ for the multiplicity-two eigenspaces, while both quantities equal one for the exceptional one-dimensional eigenspaces. Thus, for $N\geq5$, the edge relation of $C_N$ has collision dimension $\frac{2N}{3}$, while every individual Laplacian spectral projector has collision dimension at most two. The restriction $N\geq5$ excludes the exceptional graph $C_4$, whose edge collision dimension is $2$ rather than $\frac{8}{3}$.
\end{example}

\begin{example}
The Laplacian of $K_N$ has eigenvalue zero with multiplicity one and eigenvalue $N$ with multiplicity $N-1$. The projector onto the nonconstant eigenspace is $\Pi_N=I-\frac{1}{N}J$, where $I$ is the identity matrix and $J$ is the matrix all of whose entries are one, and the centered adjacency matrix from Example \ref{example:complete-graph-centered} is $\frac{1}{N}J-I=-\Pi_N$. Therefore
$$
D_4^{\circ}(K_N)=D_4(\Pi_N)=N-1,
$$
which is the spectral explanation of the centered calculation.
\end{example}

\begin{example}
Let $G=K_{d,d}$. Its Laplacian eigenvalues are $0$ and $2d$, together with $d$ when $d\geq2$, and the eigenspace for $2d$ is one-dimensional, so the formula \eqref{equation:centered-harmonic-decomposition} reduces to
$$
A-\frac{1}{2}J=-d\Pi_{2d}.
$$
Under the parity labeling, $\widehat\Pi_{2d}$ is supported at one frequency, hence
$$
D_4^{\circ}(K_{d,d})=D_4(\Pi_{2d})=\operatorname{FR}(\Pi_{2d})^2=1,
$$
and the one-dimensional centered structure found in Example \ref{example:complete-bipartite-centered} is exactly the top Laplacian eigenspace.
\end{example}

\section{Further questions}
\label{section:open-problems}

The odd prime-power classification leaves open the corresponding question at other orders.

\begin{question}[Classification at general orders]
Can the Boolean constraints in the normal form of Section~\ref{section:classification} be replaced by a structural description of all sharp graphs? In particular, what is the classification at powers of two and at orders with several distinct prime divisors?
\end{question}

The $15$-vertex example shows that a connected sharp graph with nonsingular adjacency matrix need not be vertex-transitive.

\begin{question}[Effective optimization]
For simple unweighted graphs, what is the computational complexity of deciding whether a sharp cyclic labeling exists, and what is the complexity of approximating $\operatorname{FR}_{\min}(G)$ or $E_{\mathrm{edge}}^{\max}(G)$ to prescribed accuracy?
\end{question}

Gupta and Iosevich \cite{GuptaIosevichWeighted2026} proved that approximating weighted edge complexity within additive error $\frac{1}{256N^{\frac{7}{2}}}$ is NP-hard, even on connected weighted graphs of odd order $N$, but their theorem does not settle the questions above for simple unweighted graphs. It would also be useful to bound the gap in the spectral cap for graph classes that admit no sharp labeling.

\begin{question}[Stability under additional hypotheses]
Which spectral restrictions give a dimension-free adjacency stability estimate? When can the inverse transform of an approximating weighted partial permutation matrix be replaced by the adjacency matrix of a nearby simple graph?
\end{question}

Section~\ref{section:adjacency-stability} rules out a bound depending only on the relative graph-energy deficit for arbitrary regular graphs, whereas for an orthogonal projector Theorem~\ref{theorem:projector-stability} provides such a bound, although its approximating circulant projector need not be a spectral projector of the original graph.

\begin{question}[From directed edges to graphs and observations]
Which additional hypotheses allow a large subset of the directed edge set with small difference set to yield information about vertex subsets or induced subgraphs? For a specified graph class and observation model, what observation lower bounds follow from a large necessary Fourier model size?
\end{question}

The Balog--Szemer\'edi--Gowers theorem controls a subset of directed edges and does not specify the remaining edges. One can prescribe the edges between two disjoint intervals of labels of the same length and choose the other edges at random. The prescribed directed edge set has a difference set of comparable size and remains in the final edge relation, so its presence alone carries little information about the rest of the graph.

At positive density, a bounded difference ratio alone is a weak condition, since every set $B\subset{\mathbb Z}_N^2$ of density $\beta>0$ satisfies
$$
|B-B|\leq N^2=\frac{|B|}{\beta}.
$$
Theorem~\ref{theorem:graph-bsg} is more informative for sparse graphs near the uncertainty scale $\alpha\operatorname{FR}^2=O(1)$, and to obtain a statement about induced subgraphs one must also control how the selected edges meet at their endpoints.

\appendix

\section{The quantitative Balog--Szemer\'edi--Gowers theorem}
\label{appendix:bsg}

We record the form of the Balog--Szemer\'edi--Gowers theorem used in Section \ref{section:inverse}, which is Theorem 1.2 of Reiher and Schoen \cite{ReiherSchoen2024}.

\begin{theorem}[Reiher--Schoen]
\label{theorem:bsg}
Let $A$ be a nonempty finite subset of an abelian group, and suppose that $K\geq1$, $0<\varepsilon<\frac{1}{2}$, and
$$
E_+(A)\geq\frac{|A|^3}{K}.
$$
Then there exists a subset $A'\subset A$ such that
$$
|A'|\geq(1-\varepsilon)K^{-\frac{1}{2}}|A|
\text{ and }
|A'-A'|\leq2^{33}\varepsilon^{-9}K^4|A'|.
$$
\end{theorem}

The theorem strengthens the original result of Balog and Szemer\'edi \cite{BalogSzemeredi1994} and the quantitative refinement from Gowers's work \cite{Gowers1998}.


\begin{thebibliography}{99}

\bibitem{AldalehEtAl2025}
K. Aldaleh, W. Burstein, G. Garza, G. Hart, A. Iosevich, J. Iosevich, A. Khalil, J. King, N. Kulkarni, T. Le, I. Li, A. Mayeli, B. McDonald, K. Nguyen, and N. Shaffer,
The Fourier ratio and complexity of signals,
\href{https://arxiv.org/abs/2511.19560}{arXiv:2511.19560}, 2025.

\bibitem{BalogSzemeredi1994}
A. Balog and E. Szemer\'edi,
A statistical theorem of set addition,
Combinatorica 14 (1994), no. 3, 263--268.
\href{https://doi.org/10.1007/BF01212974}{doi:10.1007/BF01212974}.

\bibitem{BandeiraLewisMixon2018}
A. S. Bandeira, M. E. Lewis, and D. G. Mixon,
Discrete uncertainty principles and sparse signal processing,
Journal of Fourier Analysis and Applications 24 (2018), no. 4, 935--956.
\href{https://doi.org/10.1007/s00041-017-9550-x}{doi:10.1007/s00041-017-9550-x}.

\bibitem{Barron1993}
A. R. Barron,
Universal approximation bounds for superpositions of a sigmoidal function,
IEEE Transactions on Information Theory 39 (1993), no. 3, 930--945.
\href{https://doi.org/10.1109/18.256500}{doi:10.1109/18.256500}.

\bibitem{BrouwerHaemers2012}
A. E. Brouwer and W. H. Haemers,
Spectra of graphs,
Universitext, Springer, New York, 2012.
\href{https://doi.org/10.1007/978-1-4614-1939-6}{doi:10.1007/978-1-4614-1939-6}.

\bibitem{BursteinIosevichNathan2026}
W. Burstein, A. Iosevich, and H. S. Nathan,
The Fourier ratio: a unifying measure of complexity for recovery,
localization, and learning,
Applied and Computational Harmonic Analysis 85 (2026),
Article 101904.
\href{https://doi.org/10.1016/j.acha.2026.101904}{doi:10.1016/j.acha.2026.101904}.

\bibitem{Chung1997}
F. R. K. Chung,
Spectral graph theory,
CBMS Regional Conference Series in Mathematics 92,
American Mathematical Society, Providence, RI, 1997.

\bibitem{Gowers1998}
W. T. Gowers,
A new proof of Szemer\'edi's theorem for arithmetic progressions of length four,
Geometric and Functional Analysis 8 (1998), no. 3, 529--551.
\href{https://doi.org/10.1007/s000390050065}{doi:10.1007/s000390050065}.

\bibitem{GuptaIosevichWeighted2026}
V. Gupta and A. Iosevich,
Edge complexity of weighted graphs: involutory symmetries and NP-hardness,
\href{https://arxiv.org/abs/2609.23842}{arXiv:2609.23842}, 2026.

\bibitem{GuptaIosevich2026}
V. Gupta and A. Iosevich,
Fourier ratios of graph kernels: energy bounds, optimal labelings, and recovery,
\href{https://arxiv.org/abs/2606.15475}{arXiv:2606.15475}, 2026.

\bibitem{GuptaIosevichEtAl2026}
V. Gupta, A. Iosevich, J. Iosevich, B. Song, and H. Tian,
Edge complexity of graphs,
\href{https://arxiv.org/abs/2607.15598}{arXiv:2607.15598}, 2026.

\bibitem{Gutman1978}
I. Gutman,
The energy of a graph,
Ber. Math.-Statist. Sekt. Forschungsz. Graz 103 (1978), 1--22.

\bibitem{LiShiGutman2012}
X. Li, Y. Shi, and I. Gutman,
Graph energy,
Springer, New York, 2012.
\href{https://doi.org/10.1007/978-1-4614-4220-2}{doi:10.1007/978-1-4614-4220-2}.

\bibitem{ReiherSchoen2024}
C. Reiher and T. Schoen,
Note on the theorem of Balog, Szemer\'edi, and Gowers,
Combinatorica 44 (2024), no. 3, 691--698.
\href{https://doi.org/10.1007/s00493-024-00092-5}{doi:10.1007/s00493-024-00092-5}.

\bibitem{Renyi1961}
A. R\'enyi,
On measures of entropy and information,
in Proceedings of the Fourth Berkeley Symposium on Mathematical Statistics and Probability,
Volume I, Contributions to the Theory of Statistics,
University of California Press, Berkeley, 1961, 547--561.

\bibitem{TaoVu2006}
T. Tao and V. H. Vu,
Additive combinatorics,
Cambridge Studies in Advanced Mathematics 105,
Cambridge University Press, Cambridge, 2006.
\href{https://doi.org/10.1017/CBO9780511755149}{doi:10.1017/CBO9780511755149}.

\bibitem{Terras1999}
A. Terras,
Fourier analysis on finite groups and applications,
London Mathematical Society Student Texts 43,
Cambridge University Press, Cambridge, 1999.
\href{https://doi.org/10.1017/CBO9780511626265}{doi:10.1017/CBO9780511626265}.

\bibitem{ZhuHayashiChen2018}
H. Zhu, M. Hayashi, and L. Chen,
Axiomatic and operational connections between the
$\ell^1$-norm of coherence and negativity,
Physical Review A 97 (2018), Article 022342.
\href{https://doi.org/10.1103/PhysRevA.97.022342}{doi:10.1103/PhysRevA.97.022342}.

\end{thebibliography}
\end{document}